\documentclass[11pt,a4paper]{article}
\usepackage{amssymb,amsmath,amsopn,amsthm,graphicx,makeidx, url}
\input xy
\xyoption{all}
\usepackage[all,2cell]{xy}
\UseAllTwocells

\begin{document}

\newtheorem{definition}{Definition}[section]
\newtheorem{definitions}[definition]{Definitions}
\newtheorem{lemma}[definition]{Lemma}
\newtheorem{prop}[definition]{Proposition}
\newtheorem{theorem}[definition]{Theorem}
\newtheorem{cor}[definition]{Corollary}
\newtheorem{cors}[definition]{Corollaries}
\theoremstyle{remark}
\newtheorem{remark}[definition]{Remark}
\theoremstyle{remark}
\newtheorem{remarks}[definition]{Remarks}
\theoremstyle{remark}
\newtheorem{notation}[definition]{Notation}
\theoremstyle{remark}
\newtheorem{example}[definition]{Example}
\theoremstyle{remark}
\newtheorem{examples}[definition]{Examples}
\theoremstyle{remark}
\newtheorem{dgram}[definition]{Diagram}
\theoremstyle{remark}
\newtheorem{fact}[definition]{Fact}
\theoremstyle{remark}
\newtheorem{illust}[definition]{Illustration}
\theoremstyle{remark}
\newtheorem{rmk}[definition]{Remark}
\theoremstyle{definition}
\newtheorem{observation}[definition]{Observation}
\theoremstyle{definition}
\newtheorem{question}[definition]{Question}
\theoremstyle{definition}
\newtheorem{conj}[definition]{Conjecture}

\newcommand{\stac}[2]{\genfrac{}{}{0pt}{}{#1}{#2}}
\newcommand{\stacc}[3]{\stac{\stac{\stac{}{#1}}{#2}}{\stac{}{#3}}}
\newcommand{\staccc}[4]{\stac{\stac{#1}{#2}}{\stac{#3}{#4}}}
\newcommand{\stacccc}[5]{\stac{\stacc{#1}{#2}{#3}}{\stac{#4}{#5}}}

\renewcommand{\marginpar}[2][]{}

\renewenvironment{proof}{\noindent {\bf{Proof.}}}{\hspace*{3mm}{$\Box$}{\vspace{9pt}}}

\title{Tensor and direct extension of definable subcategories}

\author{Mike Prest \\ Department of Mathematics, University of Manchester, UK \\ mprest@manchester.ac.uk}

\date{\today} 

\maketitle

\abstract{Definable subcategories may be extended along a ring homomorphism directly by using their defining conditions in the new module category, or by tensoring up with the new ring.  We investigate what is preserved and reflected by these processes.  On the way, we introduce the notion of being Mittag-Leffler with respect to a bimodule - a refinement of the Mittag-Leffler condition.   Particular attention is given to the case where the ring homomorphism is an elementary embedding.}

\tableofcontents

\section{Introduction}

A subcategory of a module category ${\rm Mod}\mbox{-}R$ is said to be definable if it is closed in ${\rm Mod}\mbox{-}R$ under direct products, directed colimits and pure submodules, equivalently if it is an additive axiomatisable subcategory of ${\rm Mod}\mbox{-}R$.  If $f:R \to S$ is a homomorphism of rings, then there are two obvious ways of inducing a definable subcategory of ${\rm Mod}\mbox{-}S$ from a definable subcategory ${\cal D}$ of ${\rm Mod}\mbox{-}R$:  {\em direct extension}, which takes an axiomatisation of ${\cal D}$, transfers it along $f$ and then applies this in ${\rm Mod}\mbox{-}S$; {\em tensor extension}, which is the smallest definable subcategory of ${\rm Mod}\mbox{-}S$ which contains all modules of the form $M\otimes_RS_S$ with $M \in {\cal D}$.

We consider in Sections \ref{secdirextn} and \ref{sectensextn} the effects of these extension processes, paying particular attention to tensor extension, indeed considering more generally the effect that tensoring with a bimodule $_RB_S$ has on definable subcategories of ${\rm Mod}\mbox{-}R$.

This leads us to consider a strengthening of the relative Mittag-Leffler condition.  Recall that a right $R$-module $M_R$ is Mittag-Leffler if, for every set $\{\,_RL_i\}_i$ of left $R$-modules, the natural map $M\otimes_R\prod_iL_i \to \prod_i \, M\otimes_RL_i$ is monic.  There is a relative version of the definition under which the modules $L_i$ are restricted to those belonging to some definable subcategory of $R\mbox{-}{\rm Mod}$ (the category of left $R$-modules), see Section \ref{secML}.  If the $L_i$ all are $(R,S)$-bimodules for some ring $S$, then that natural map, even if monic, need not, in the relativised case, be pure as an embedding of $S$-modules, see \ref{nontriv}.  We characterise in \ref{MLpure} those modules $M$ where tensoring with $M$ does give a pure embedding of $S$-modules.

We consider various conditions under which we can say more about extending definable subcategories along a morphism of rings.  In particular, in Section \ref{secextindelem} we consider the case where $f:R\to S$ is an elementary embedding of rings.

\vspace{4pt}

Subcategories are assumed to be full and closed under isomorphism.  Rings $R$ can be taken to be rings with one object but everything works with $R$ a (skeletally) small preadditive category.  In that more general case, elements of modules and variables in formulas are sorted which, in some arguments, would require more cumbersome notation.  Therefore we present the results for the usual, 1-sorted ring, case, but the arguments work just as well over rings with many objects (see, for instance, \cite{PreErice}).  The model theory that we use is summarised in the next section and always refers to the usual language for $R$-modules.  One can see various other references for details about this, but perhaps enough is said here to make the paper understandable without the need to refer elsewhere.

\section{Background} \label{secbkgnd} \marginpar{secbkgnd}

This section is rather dense with information and notation but more leisurely introductions can be found in many other papers.  We use \cite{PreNBK} as a reference which gathers together many of the results that we will need.

\subsection{Pp formulas, pp-pairs and pp-types}

A {\bf pp formula} for right $R$-modules is an existentially quantified (homogeneous) system of $R$-linear equations, that is, a formula of the form 
$$\exists \overline{y} \, (\overline{x}A + \overline{y}B =0),$$
where $A$ and $B$ are suitable-sized matrices with entries from $R$ and tuples $\overline{x}$, $\overline{y}$ of variables are read as row vectors (and $0$ is a column vector).  We unpack this compact notation to write the formula more explicitly as 
$$\exists \overline{y} \, \bigwedge_{j=1}^m \, \sum_{i=1}^n x_ir_{ij} +\sum_{k=1}^t y_ks_{kj} =0.$$
Here the $r_{ij}$ and $s_{kj}$ - the entries of $A$ and $B$ - are elements\footnote{precisely, function symbols standing for multiplication by those elements; for details of how the formal language is set up, see elsewhere} of $R$ and $\bigwedge $ is to $\wedge$ (``and") as $\sum$ is to $+$; so this is a system of $m$ $R$-linear equations prefaced by a quantifier which has the effect of projecting (solution sets) along the $\overline{y}$-axis.  The variables $\overline{x} = (x_1, \dots, x_n)$ are free to be substituted with values from some $R$-module.  If we denote this formula by $\phi$, then we may write $\phi(\overline{x})$ or $\phi(x_1,\dots, x_n)$ to display its free = unquantified variables.

Given such a pp formula $\phi$ we have, in each module $M$,  its {\bf solution set}:  
$$\phi(M) = \{ \overline{a} \in M^n: \exists \overline{b} \in M^t \text{ such that }  \overline{a}A + \overline{b}B =0 \}.$$
We say that $\phi(M)$ is a subgroup of $M^n$ {\bf pp-definable} in $M$ or, more briefly, a {\bf pp-definable subgroup} of $M$.  The notation $M \models \phi(\overline{a})$, which is more usual in model theory, is equivalent to  $\overline{a} \in \phi(M)$.

If $M$ is a submodule of $N$, $M\leq N$, then $M$ is a {\bf pure submodule} if, for every tuple $\overline{a}$ from $M$, if $N\models \phi(\overline{a})$, then $M\models \phi(\overline{a})$.  More generally, a {\bf pure monomorphism} is a monomorphism whose image is pure in its codomain.

If $\psi$ and $\phi$ are formulas with the same free variables $x_1, \dots, x_n$, then we write $\psi \leq \phi$ if, for every module $M$, $\psi(M) \leq \phi(M)$, that is, $\psi(M)$ is a subgroup of $\phi(M)$ (pp-definable subgroups are modules over the centre of the ring $R$ but are not, in general, $R$-submodules).  Formulas which are equivalent in this preordering, for instance $\phi(\overline{x})$ and $\phi(\overline{x}) \wedge (\overline{x} = \overline{x})$, have the same solution set in every module and in practice are identified, so we regard this as an ordering on (equivalence classes of) pp formulas.  The resulting poset, denoted ${\rm pp}^n_R$, is a lattice, with meet given by conjunction, $\phi \wedge \psi$, and join given by $\phi + \psi$, where the latter is the pp formula $\exists \overline{x}_1 \overline{x}_2 \,( \phi(\overline{x}_1)\, \wedge \,\psi(\overline{x}_2) \,\wedge \,\overline{x} = \overline{x}_1 + \overline{x}_2)$.  These define $\phi(M) \cap \psi(M)$, respectively $\phi(M)+ \psi(M)$, in any module $M$.  One may note that the conjunction and sum of two pp formulas are not strictly of the form given in the definition of pp formula above but they may be rewritten as equivalent formulas in that form, so we do refer to them as pp formulas.

A {\bf pp-pair} is a pair $\psi \leq \phi$ of pp-formulas, that is, where one implies the other.  We notate such a pair as $\phi/\psi$ because often we are interested in the factor groups $\phi(M)/\psi(M)$.  We say that a pp-pair $\phi / \psi$ is {\bf closed} on a module $M$ if $\phi(M) = \psi(M)$; otherwise it is {\bf open} on $M$.  

The {\bf pp-type of} an element $a$ in a module $M$ is the set of all pp formulas that $a$ satisfies in $M$; more generally for $n$-tuples:\footnote{The free variables $x_1, \dots x_n$ are just placeholders, so we keep them in the background in our notation whenever we can.} 
$${\rm pp}^M(\overline{a}) = \{ \phi(\overline{x}): M \models \phi(\overline{a}) \}.$$
We say that $\overline{a}$ is a {\bf realisation} of that pp-type {\bf in} $M$.  Every set $p$ of pp formulas which is a {\bf filter}, that is, upwards-closed (if $\phi \leq \psi$ and $\phi\in p$ then $\psi \in p$) and closed under intersection/conjunction ($\phi, \psi \in p$ implies $\phi \wedge \psi \in p$), occurs in this way, so we refer to such a set as a {\bf pp-type}.

A pp-type $p$ is {\bf finitely generated} if there is a pp formula $\phi \in p$ such that $p=\{ \psi: \phi \leq \psi\}$; we write $p =\langle \phi\rangle$ and say that $p$ is {\bf generated by} $\phi$.  If $A$ is finitely presented and $\overline{a}$ is from $A$, then ${\rm pp}^A(\overline{a})$ is finitely generated, \cite[1.2.6]{PreNBK}.

Morphisms preserve pp formulas: if $f:M \to N$ and $M\models \phi(\overline{a})$, then $N \models \phi(f\overline{a})$.  Thus morphisms are non-decreasing on pp-types:  ${\rm pp}^M(\overline{a}) \subseteq {\rm pp}^N(f\overline{a})$ if $f$ is as above.  Note that $f$ is a pure monomorphism iff ${\rm pp}^M(\overline{a}) = {\rm pp}^N(f\overline{a})$ for every $\overline{a}$ from $M$.

\subsection{Pure-injectives and the Ziegler spectrum}\label{secpizg} \marginpar{secpizg}

A module $N$ is {\bf pure-injective} if it is injective over pure monomorphisms, equivalently if every embedding of $N$ as a pure submodule is a split embedding. 

The (right) {\bf Ziegler spectrum}, \cite{Zie}, ${\rm Zg}_R$ of $R$ is the topological space on the set, ${\rm pinj}_R$, of isomorphism classes of indecomposable pure-injective (right) $R$-modules and which has, for a basis of open sets, those of the form
$$(\phi/\psi)  = \{  N\in {\rm Zg}_R: \phi(N)/\psi(N) \neq 0\}.$$

If ${\cal T}$ is any topological space then the lattice ${\cal O}({\cal T})$ of open subsets of ${\cal T}$ is a complete Heyting algebra (a complete lattice where meet distributes over arbitrary join), also referred to as a {\bf frame}.

\subsection{Definable subcategories of module categories} \label{secdefsub} \marginpar{secdefsub}

We say that ${\cal D}$ is a {\bf definable subcategory} of ${\rm Mod}\mbox{-}R$ if it satisfies the following equivalent conditions, where a subcategory is {\bf additive} if it is closed under finite direct sums and under direct summands.

\begin{theorem}\label{defchar} \marginpar{defchar} (see, e.g., \cite[3.4.7]{PreNBK} or \cite[9.1]{PreErice}) The following conditions on a subcategory ${\cal D}$ of ${\rm Mod}\mbox{-}R$ are equivalent:

\noindent (i)  ${\cal D}$ is additive and axiomatisable by some set of sentences in the language for $R$-modules;

\noindent (ii)  there is a  set $\Phi$ of pp-pairs such that ${\cal D} = \{ M \in {\rm Mod}\mbox{-}R: \phi(M) =\psi(M) \text{ for all } \phi/\psi \in \Phi\}$;

\noindent (iii)  ${\cal D}$ is closed in ${\rm Mod}\mbox{-}R$ under direct products, directed colimits and pure submodules;

\noindent (iv)  ${\cal D}$ is closed in ${\rm Mod}\mbox{-}R$ under direct products, ultraproducts (see Section \ref{secup} for these) and pure submodules.
\end{theorem} 

We say that ${\cal D}$ is {\bf defined} by $\Phi$ in the situation of (ii) above.

A definable subcategory of ${\rm Mod}\mbox{-}R$ is also closed under images of pure epimorphisms and under pure-injective hulls (e.g.~\cite[3.4.8]{PreNBK}).

If $M$ is a module then we denote by $\langle M \rangle$ the definable subcategory {\bf generated} by $M$ - the smallest definable subcategory (of the ambient module category) containing $M$.  Therefore $\langle M \rangle$ consists of the class of modules $N$ such that every pp-pair closed on $M$ is closed on $N$:
$$ \langle M \rangle = \{ N\in {\rm Mod}\mbox{-}R: \phi(M) = \psi(M) \implies \phi(N) =\psi(N) \text{ for all }  \phi \geq \psi \text{ pp }\}.$$
Similar notation is used for the definable subcategory generated by a class of modules.
Every definable subcategory is generated by some, by no means unique, $M$.  We say that two $R$-modules $M, N$ are {\bf definably equivalent}\footnote{This condition appears in \cite{PreBk} as the requirement that ${\rm Th}(M^{\aleph_0}) = {\rm Th}(N^{\aleph_0})$, where ${\rm Th}(M)$ is the complete theory of $M$} if they generate the same definable subcategory of ${\rm Mod}\mbox{-}R$: $\langle M \rangle = \langle N \rangle$, equivalently if every pp-pair open on $M$ is open on $N$ and {\it vice versa}.  

\vspace{4pt}

The set of definable subcategories of a module category ${\rm Mod}\mbox{-}R$ is closed under arbitrary intersections.

\begin{lemma}\label{infmeetdef} \marginpar{infmeetdef} (e.g.~proof of \cite[5.1.1]{PreNBK}) Suppose that ${\cal D}_i$, $i\in I$, are definable subcategories of ${\rm Mod}\mbox{-}R$, with  ${\cal D}_i$ being defined by $\Phi_i$.  Then their intersection $\bigcap_i\, {\cal D}_i $ is definable, being defined by $\bigcup_i\, \Phi_i$. 
\end{lemma}

\begin{lemma}\label{infjoindef} \marginpar{infjoindef}   Suppose that ${\cal D}_i$, $i\in I$, are definable subcategories of ${\rm Mod}\mbox{-}R$ and let $\Phi_i$ be the set of all pp-pairs closed on ${\cal D}_i$.  Then the definable join ${\cal D} = \langle \bigcup_i\, {\cal D}_i \rangle$ is defined by $\bigcap_i\, \Phi_i$.
\end{lemma}
\begin{proof}  Certainly the category defined by $\Phi = \bigcap_i\, \Phi_i$ contains each ${\cal D}_i$ and hence contains ${\cal D}$ (since that is the smallest definable subcategory which contains each ${\cal D}_i$).  For the converse, if ${\cal D}$ were strictly contained in the subcategory defined by $\Phi$, then there would be some pp-pair $\phi/\psi$ not in $\Phi$ but closed on ${\cal D}$.  There is some $i$ such that $\phi/\psi \notin \Phi_i$ so, by definition of $\Phi_i$, there is some $D\in {\cal D}_i$ such that $\phi/\psi$ is open on $D$.  But $D\in {\cal D}$ and that is a contradiction.  Therefore $\Phi$ defines precisely ${\cal D}$.
\end{proof}

Therefore the set of definable subcategories of ${\rm Mod}\mbox{-}R$, ordered by inclusion, is a complete lattice which we denote by ${\rm DefSub}_R$.

\begin{theorem}\label{Zgdef} \marginpar{Zgdef} (e.g.~\cite[5.1.1, 5.1.5]{PreNBK}) The lattice ${\rm DefSub}_R$ of definable subcategories of ${\rm Mod}\mbox{-}R$ is naturally isomorphic to the lattice of closed subsets of the Ziegler spectrum of $R$ and so is anti-isomorphic to the frame ${\cal O}({\rm Zg}_R)$ of open subsets of ${\rm Zg}_R$.
\end{theorem}

The first correspondence in \ref{Zgdef} is ${\cal D} \mapsto {\cal D} \cap {\rm pinj}_R$ and, in the other direction, it takes a set of indecomposable pure-injective modules to the definable subcategory that they together generate.  We write ${\rm Zg}({\cal D})$ for the set ${\cal D} \cap {\rm pinj}_R$ given the relative topology induced from ${\rm Zg}_R$.

The join of finitely many definable subcategories is described algebraically as follows.

\begin{lemma}\label{finjoindef} \marginpar{finjoindef} \cite[3.4.9]{PreNBK} Suppose that ${\cal D}_1, \dots, {\cal D}_n$, are definable subcategories of ${\rm Mod}\mbox{-}R$.  Then their join ${\cal D} = \langle  {\cal D}_1 \cup \dots \cup {\cal D}_n \rangle$ consists of the pure submodules of modules of the form $M_1 \oplus \dots \oplus M_n$ where $M_i \in {\cal D}_i$.
\end{lemma}

The description of \ref{finjoindef} does not work for infinite joins, see \cite[3.12]{PreDefMon}.

\vspace{4pt}

If ${\cal D}$ is a definable category, then a pure-injective $N\in {\cal D}$ is an {\bf elementary cogenerator} for ${\cal D}$ if, for each $M \in {\cal D}$, there is a pure embedding $M \to N^I$ for some index set $I$.  Every definable category has an elementary cogenerator (see \cite[5.3.52]{PreNBK}).  We will use the characterisation, see \cite[3.8]{PreDefMon}, that a pure-injective module $N$ is an elementary cogenerator (for the definable subcategory that it generates) iff every ultrapower of $N$ purely embeds in some direct power of $N$ (for the definition of ultrapower, see Section \ref{secup} below).

\begin{cor}\label{elemcogjoin} \marginpar{elemcogjoin}  If $N_i$ is an elementary cogenerator for the definable subcategory ${\cal D}_i$, $i=1, \dots, n$, then $N_1 \oplus \dots \oplus N_n$ is an elementary cogenerator for $\langle  {\cal D}_1 \cup \dots \cup {\cal D}_n \rangle$.
\end{cor}

If ${\cal X}$ is any subclass or subcategory of ${\rm Mod}\mbox{-}R$, then we write $\psi \leq_{\cal X} \phi$ if $\psi(M) \leq \phi(M)$ for every $M \in {\cal X}$, and we write $\psi =_{\cal X} \phi$ if $\psi(M) = \phi(M)$ for every $M \in {\cal X}$.  If ${\cal X} = {\rm Mod}\mbox{-}R$ then we drop the subscript, identifying pp formulas that agree on all $R$-modules.

\begin{remark}\label{ordXeqordD} \marginpar{ordXeqordD} It follows directly that, if ${\cal X}$ is any subclass or subcategory of ${\rm Mod}\mbox{-}R$ and ${\cal D} = \langle {\cal X} \rangle$ is the definable subcategory that ${\cal X}$ generates, then $\leq_{\cal D}$ and $\leq_{\cal X}$ coincide, as do $=_{\cal D}$ and $=_{\cal X}$.
\end{remark}

That follows from \ref{defchar}.

The relation, $ =_{\cal D}$, of ${\cal D}$-equivalence between pp formulas in a given set of, say $n$, free variables is a congruence on the lattice ${\rm pp}^n_R$ of pp formulas in those $n$ free variables.  Therefore, there is an induced surjective homomorphism of modular lattices ${\rm pp}^n_R \to {\rm pp}^n({\cal D})$, where the latter is defined to be the ordered set of pp formulas, in $n$ free variables, modulo equivalence $=_{\cal D}$ on ${\cal D}$, equally modulo equality when evaluated on $M$ if $\langle M \rangle = {\cal D}$.

\begin{lemma}\label{ppclosed} \marginpar{ppclosed} If $\overline{a}$ is from $M\in {\cal D}$, then ${\rm pp}^M(\overline{a})$ is closed under $=_{\cal D}$ and upwards closed with respect to $\leq _{\cal D}$.  
\end{lemma}

If $\overline{a}$ is from $M$ and ${\cal X}$ is any class or category of modules, then we say that ${\rm pp}^M(\overline{a})$ is ${\cal X}${\bf -generated} by $\phi$, and write ${\rm pp}^M(\overline{a}) =_{\cal X} \langle \phi \rangle$, if ${\rm pp}^M(\overline{a}) = \{ \psi: \phi \leq_{\cal X} \psi\}$, in which case we say that ${\rm pp}^M(\overline{a})$ is ${\cal X}$-{\bf finitely generated}.  By \ref{ordXeqordD} we can suppose that ${\cal X}$ is a definable subcategory.  We don't assume in this definition that $M\in \langle{\cal X}\rangle$.

If ${\cal X}$ is a class of right $R$-modules then a right $R$-module $M$ is ${\cal X}${\bf -atomic} iff, for every finite tuple $\overline{a}$ from $M$, the pp-type, ${\rm pp}^M(\overline{a})$, of $\overline{a}$ in $M$ is ${\cal X}$-finitely generated.  Again, by \ref{ordXeqordD}, we may assume that ${\cal X}$ is a definable subcategory.

\subsection{Elementary duality}

The ({\bf elementary}) {\bf dual} (\cite{PreDual}) of a pp formula $\phi$
$$\exists \overline{y} \, (\overline{x}A + \overline{y}B =0)$$ 
for right modules is the pp formula $D\phi$ for left modules which is 
$$\exists \overline{z} \, (\overline{x} = A\overline{z}\, \wedge \, B\overline{z}=0).$$  

For example, the dual of the annihilator formula $xr=0$ for right $R$-modules is (up to equivalence) the divisibility formula $r|x$, meaning $\exists y \, (x=ry)$, for left $R$-modules, and {\it vice versa}.

\begin{theorem}\label{elemdual} \marginpar{elemdual} (see \cite[\S 1.3.1]{PreNBK}) For each $n$, elementary duality $D(-)$ of pp formulas defines an anti-isomorphism ${\rm pp}^n_R \simeq \big( {\rm pp}^n_{R^{\rm op}}\big)^{\rm op}$.  

\noindent In particular $D(\psi \wedge \phi) = D\psi + D\phi$, $D(\psi + \phi) = D\psi \, \wedge \, D\phi$ and $D^2\phi = \phi$ (where $D$ is also used to denote elementary duality from left to right formulas).
\end{theorem}

If ${\cal D}$ is a definable subcategory of ${\rm Mod}\mbox{-}R$, determined by closure of some set $\Phi$ of pp-pairs, then the ({\bf elementary}) {\bf dual definable category} ${\cal D}^{\rm d}$ (\cite{HerzDual}) is the subcategory of $R\mbox{-}{\rm Mod}$ defined by the set of dual pairs - the collection of $D\psi/D\phi$ such that $\phi/\psi \in \Phi$.  This is well-defined:  the result ${\cal D}^{\rm d}$ is the same for any choice of $\Phi$ defining ${\cal D}$ (see \cite[3.4.16]{PreNBK}).  We have, e.g.~\cite[3.4.18]{PreNBK}, $({\cal D}^{\rm d})^{\rm d} ={\cal D}$.  

If $(-)^\ast$ is a duality from ${\rm Mod}\mbox{-}R$ to $R\mbox{-}{\rm Mod}$ such as $M^\ast = {\rm Hom}_{\mathbb Z}(M,{\mathbb Q}/{\mathbb Z})$, or $M^\ast = {\rm Hom}_K(M,K)$ if $R$ is an algebra over a field $K$, then $M\in {\cal D}$ iff $M^\ast \in {\cal D}^{\rm d}$, e.g.~see \cite[3.4.17]{PreNBK}.

\begin{prop}\label{dualord} \marginpar{dualord} (\cite[2.6]{HerzDual}) Let ${\cal D}$ be a definable subcategory of ${\rm Mod}\mbox{-}R$ and let ${\cal D}^{\rm d}$ be its dual definable category.  If $\phi$ and $\psi$ are pp formulas in the same free variables, then
$$\psi \leq _{\cal D} \phi \text{ iff  } D\phi \leq_{{\cal D}^{\rm d}} D\psi.$$
\end{prop}

\subsection{Ultrapowers and ultraproducts} \label{secup} \marginpar{secup}

If $I$ is a set then an {\bf ultrafilter} on $I$ is a collection ${\cal U}$ of subsets of $I$ which is a filter in the lattice of subsets of $I$ and which is such that, for every $J\subseteq I$, either $J\in {\cal U}$ or $I\setminus J \in {\cal U}$.

Given a structure $X$ - a ring, a module over a ring, ... (precisely, a structure for some first order finitary language) - we can define the corresponding {\bf ultrapower} $X^I/{\cal U}$ of $X$ to be the structure whose underlying set is $X^I/\sim_{\cal U}$ where $\sim_{\cal U}$ is the equivalence relation on $X^I$ given by $(a_i)_i \sim_{\cal U} (b_i)_i$ iff $\{ i\in I: a_i = b_i\} \in {\cal U}$.  We denote the equivalence class of $(a_i)_i$ by $(a_i)_i /{\cal U}$.  The structure on $X^I/{\cal U}$ is defined pointwise; let's be more precise supposing that $X$ is a ring.  First though, we remark that the idea is that sets in ${\cal U}$ are ``large" so $\sim_{\cal U}$ identifies elements of the product which are equal on a ``large" set of coordinates.

If $X$ is a ring, then we define addition on the ultrapower $X^I/{\cal U}$ by $(a_i)_i /{\cal U} + (b_i)_i/{\cal U} = (a_i+b_i)_i/{\cal U}$ and similarly for multiplication.  The identity and zero are $(1_i)_i/{\cal U}$ and $(0_i)_i/{\cal U}$ respectively.  It follows easily from the definition of ultrafilter that the subset $Z \subseteq X^I$ of elements which are $\sim_{\cal U}$-equivalent to $0$ is a two-sided ideal of the ring $X^I$ and that, as a set and as a ring, $X^I/{\cal U}$ is exactly the factor ring $X^I/Z$.

Similarly for modules:  if $M$ is a right $R$-module, then $M^I/{\cal U}$ is a right $R$-module with $(a_i)_i/{\cal U} \cdot r$ defined to be $(a_ir)_i/{\cal U}$.  In fact, defining $(a_i)_i/{\cal U} \cdot (r_i)_i/{\cal U} = (a_ir_i)_i/{\cal U}$ makes $M^I/{\cal U}$ a right $R^I/{\cal U}$-module with the $R$-action on $M^I/{\cal U}$ arising {\it via} the diagonal embedding (see \ref{elemextup}) $R \to R^I/{\cal U}$.

More generally, given a set $\{ X_i\}_{i\in I}$ of structures (all of the same kind) and an ultrafilter ${\cal U}$ on $I$ we can form, in the same way, the {\bf ultraproduct} $\prod_{i\in I} X_i/{\cal U}$.  The relationship between the properties of an ultrapower or ultraproduct and those of its factors is as follows.

\begin{theorem} \label{Los}\marginpar{Los} (\L os' Theorem) (e.g.~\cite[9.5.1]{Hod}) Suppose that $\overline{a} = (\overline{a_i})_i/{\cal U}$ is a tuple from $\prod_{i\in I} X_i/{\cal U}$ and let $\phi(\overline{x})$ be a formula (in a language appropriate for these structures $X_i$).  Then 
$$\overline{a} \in \phi(\prod_{i\in I} X_i/{\cal U}) \text{ iff } \{ i\in I: \overline{a_i}\in \phi(X_i)\} \in {\cal U}.$$
\end{theorem}

We allow the case that the tuple is empty, in which case $\phi$ is just a sentence - a formula with no free variables.

Clearly the construction may be done with any filter, in which case we have a {\bf reduced power} or {\bf reduced product}.  In that more general case, the conclusion of \L os' Theorem still holds true for pp formulas $\phi$; this follows from the proof of \L os' Theorem, which may be found in most books on model theory.

\section{Direct extension of definable subcategories} \label{secdirextn} \marginpar{secdirextn}

Suppose that $f:R\to S$ is a morphism of rings.  If $\phi$ is a pp formula for $R$-modules, then we denote by $f_\ast\phi$ the pp formula for $S$-modules obtained by replacing each ring element $r$ appearing in $\phi$, by $fr$.  Note that $f_\ast (\phi \wedge \psi) = f_\ast\phi \wedge f_\ast\psi$ and $f_\ast (\phi + \psi) = f_\ast\phi + f_\ast\psi$, so this is a lattice homomorphism from ${\rm pp}^n_R$ to ${\rm pp}^n_S$.

\begin{lemma}\label{fstar} \marginpar{fstar} \cite[3.2.7]{PreNBK} If $M \in {\rm Mod}\mbox{-}S$, then $f_\ast\phi(M_S) = \phi(M_R)$.
\end{lemma}

That is, the meaning of $f_\ast\phi$ in an $S$-module $M$ is the same as the meaning of $\phi$ in $M$ regarded as an $R$-module {\it via} $f$.  We have the following corollaries.

\begin{cor}\label{fstarleq} \marginpar{fstarleq} If $\phi/\psi$ is a pp-pair for (right) $R$-modules, then $f_\ast \phi / f_\ast \psi$ is a pp-pair for (right) $S$-modules.
\end{cor}

\begin{cor}\label{fstarequiv} \marginpar{fstarequiv} If $\phi$ and $\psi$ are equivalent pp formulas for (right) $R$-modules, then $f_\ast \phi$ and $f_\ast \psi$ are equivalent for (right) $S$-modules.
\end{cor}

\begin{cor} \label{fstarclosed} \marginpar{fstarclosed} If $\phi$ and $\psi$ are pp formulas for $R$-modules and $M\in {\rm Mod}\mbox{-}S$, then the pp-pair $f_\ast\phi/f_\ast \psi$ is closed on $M_S$ iff the pair $\phi/\psi$ is closed on $M_R$.
\end{cor}

So, in practice, particularly if $R$ is a subring of $S$ {\it via} $f$, we often just identify $\phi$ and $f_\ast \phi$.

Let ${\cal D}$ be a definable subcategory of ${\rm Mod}\mbox{-}R$.  Define the {\bf direct extension} of ${\cal D}$ to ${\rm Mod}\mbox{-}S$, to be
$${\cal D}^S = \{M_S: M_R \in {\cal D}\}.$$
(We could more accurately notate this as ${\cal D}^f$.)
Since restriction commutes with direct products and directed colimits, and since an $S$-pure embedding is, in particular, $R$-pure, by \ref{defchar} ${\cal D}^S$ is a definable subcategory of ${\rm Mod}\mbox{-}S$.  Indeed it is defined by closure of ``the same" pp-pairs that define ${\cal D}$.

\begin{lemma}\label{ppdirextn} \marginpar{ppdirextn} If ${\cal D}$ is a definable subcategory of ${\rm Mod}\mbox{-}R$ and if $\Phi$ is a set of pp-pairs for $R$-modules defining (by their closure) ${\cal D}$, then $f_\ast\Phi = \{ f_\ast\phi/f_\ast \psi: \phi/ \psi \in \Phi \}$ defines the direct extension ${\cal D}^S$ of ${\cal D}$ to a definable subcategory of ${\rm Mod}\mbox{-}S$.
\end{lemma}
\begin{proof} If $M\in {\cal D}^S$ and $\phi/ \psi$ is closed on ${\cal D}$ then, since $M_R\in {\cal D}$, we have that $\phi/\psi$ is closed on $M_R$.  So, by \ref{fstarclosed}, $f_\ast\phi / f_\ast\psi$ is closed on $M_S$.

For the converse, if $M_S$ is such that, for every pp-pair $\phi/\psi$ in $\Phi$ we have $f_\ast\phi /f_\ast\psi$ closed on $M_S$ then, by \ref{fstarclosed}, each $\phi/\psi \in \Phi$ is closed on $M_R$.  It follows that $M_R \in {\cal D}$, so $M\in {\cal D}^S$, as required.
\end{proof}

\begin{example} \label{exindrestr} \marginpar{exindrestr} If ${\cal D}$ is a definable subcategory of ${\rm Mod}\mbox{-}R$ then the class of $R$-modules which are restrictions of $S$-modules in ${\cal D}^S$ may be strictly smaller than ${\cal D}$.  For instance, take the ring morphism $R \to S$ to be ${\mathbb Q} \to {\mathbb R}$ and ${\cal D}$ to be all of ${\rm Mod}\mbox{-}{\mathbb Q}$, so ${\cal D}^S ={\rm Mod}\mbox{-}{\mathbb R}$.

Examples of the form $\pi_1:R_1 \times R_2 \to R_1$ show how this can happen for even simpler reasons.

We now take account of this.
\end{example}

Suppose that $f:R \to S$ is a ring homomorphism and that ${\cal C}$ is a definable subcategory of ${\rm Mod}\mbox{-}S$.  Set $f^\ast {\cal C} = \{ M_R: M \in {\cal C}\}$.  Let $\Phi({\cal C})$ be the set of all pp-pairs for $S$-modules which are closed on ${\cal C}$.  Define the ({\bf definable}) {\bf restriction}, ${\cal C}|_R$, of ${\cal C}$ to ${\rm Mod}\mbox{-}R$ along $f$ to be the subcategory of ${\rm Mod}\mbox{-}R$ defined by the subset $\Phi({\cal C})_R$ of $\Phi({\cal C})$ consisting of all those pp-pairs $\phi/\psi$ for $R$-modules such that $f_\ast \phi /f_\ast \psi \in \Phi({\cal C})$ (that is, we use all the pp-pairs in $\Phi({\cal C})$ which can be applied to $R$-modules).

In the special case ${\cal C} = {\rm Mod}\mbox{-}S$, we denote ${\cal C}|_R$ by ${\rm DefTr}^S_R$, or ${\rm DefTr}(f)$, the {\bf definable trace} of ${\rm Mod}\mbox{-}S$ in ${\rm Mod}\mbox{-}R$.

\begin{lemma}\label{restrdef} \marginpar{restrdef} Suppose that $f:R \to S$ is a ring homomorphism and that ${\cal C}$ is a definable subcategory of ${\rm Mod}\mbox{-}S$.  Then $f^\ast {\cal C} = \{ M_R: M \in {\cal C}\} \subseteq {\cal C}|_R$, and the definable subcategory of ${\rm Mod}\mbox{-}R$ generated by $f^\ast {\cal C}$ is ${\cal C}|_R$; that is, $\langle f^\ast {\cal C} \rangle = {\cal C}|_R$.
\end{lemma}
\begin{proof} The first statement is immediate by \ref{fstarclosed}.  If the containment $\langle f^\ast {\cal C} \rangle \subseteq {\cal C}|_R$ were proper, then there would be a pp-pair $\phi/\psi$ for $R$-modules closed on $f^\ast {\cal C}$ but not on ${\cal C}|_R$.  The former implies, by \ref{fstarclosed} again, that $f_\ast \phi /f_\ast \psi$ is closed on ${\cal C}$, hence that $f_\ast \phi /f_\ast \psi \in \Phi({\cal C})$.  Therefore, by definition, $\phi/\psi \in \Phi({\cal C})_R$ and so that pair is closed on ${\cal C}|_R$, contradiction, as required.
\end{proof}

We give an algebraic description of ${\cal C}|_R$ in \ref{restr}, using the following result.

\begin{lemma}\label{restelemcog} \marginpar{restelemcog}  Suppose that $R \to S$ is a ring homomorphism and ${\cal C}$ is a definable subcategory of ${\rm Mod}\mbox{-}S$.  If $N$ is an elementary cogenerator for ${\cal C}$, then $N_R$ is an elementary cogenerator for ${\cal C}|_R$.
\end{lemma}
\begin{proof}  First we show that $N_R$ is an elementary cogenerator for the definable category $\langle N_R \rangle$ that it generates.

Since $N$ is pure-injective as an $S$-module it is pure-injective as an $R$-module (e.g.~\cite[4.3.7]{PreNBK}).  So, by \cite[3.8]{PreDefMon}, it is enough to show that every ultrapower of $N_R$ purely embeds in a power of $N_R$.  Consider an ultrapower $(N_R)^I/{\cal U}$ where ${\cal U}$ is an ultrafilter on a set $I$.  We can form the corresponding ultrapower $N^I/{\cal U}$ of the $S$-module $N$.  Since $N$ is an elementary cogenerator, there is an $S$-pure embedding $N^I/{\cal U} \to N^\kappa$ for some $\kappa$.  Note that $(N^I/{\cal U})_R$ is exactly $(N_R)^I/{\cal U}$ (form the ultrapower, then apply the module restriction functor $f^\ast$ or apply $f^\ast$, then form the ultrapower; that these are the same can be seen directly or, since the restriction functor has both a left and right adjoint it commutes with powers and directed colimits, hence (e.g.~\cite[3.3.1]{PreNBK}, with ultrapowers).  Therefore, on restriction, we have an $R$-pure embedding $(N_R)^I/{\cal U} \to (N_R)^\kappa$, as required.

Just for interest, we note a different, more model-theoretic, proof which uses a different characterisation of elementary cogenerators - as pure-injectives which realise every neg-isolated type (see \cite[5.3.50]{PreNBK}).  So we have to show that if $p$ is a pp-type for $\langle N_R \rangle$ which is neg-isolated, say by the pp formula $\psi$, then $p$ is realised in $N$.  So we consider $f_\ast p$ - a set of pp formulas for $S$-modules, which we can extend to a set $q$ of pp formulas which is maximal, relative to $\langle N \rangle = {\cal C}$, not containing $f_\ast \psi$.  Since $N$ is an elementary cogenerator for ${\cal C}$, there is $\overline{a}$ from $N$ such that ${\rm pp}^N(\overline{a}) = q$.  Then ${\rm pp}^{N_R}(\overline{a}) = p$, as required.

So $N_R$ is an elementary cogenerator for $\langle N_R \rangle$.  Since each $M \in {\cal C}$ purely embeds in some power $N^\kappa$, each $M_R$ purely embeds in some power $(N_R)^\kappa$ and hence $\langle N_R \rangle = \langle f^\ast {\cal C} \rangle = {\cal C}|_R$.
\end{proof}

\begin{cor} \label{restr} \marginpar{restr} Suppose that $R \to S$ is a ring homomorphism and that ${\cal C}$ is a definable subcategory of ${\rm Mod}\mbox{-}S$.  Then 
$${\cal C}|_R = \{M'\in {\rm Mod}\mbox{-}R: M' \text{ is a pure submodule of } M_R \text{ for some } M\in {\cal C}\}.$$ 
\end{cor}

That is, to obtain ${\cal C}|_R$ from ${\cal C}$, we just have to close $f^\ast{\cal C}$ under ($R$-)pure submodules.

\vspace{4pt}

Given a definable subcategory ${\cal D}$ of ${\rm Mod}\mbox{-}R$, set 
$${\cal D}^S_R = ({\cal D}^S)|_R.$$
Directly from the definitions we have the following.

\begin{lemma}\label{deftrextn} \marginpar{deftrextn} Suppose that $R \to S$ is a ring homomorphism and that ${\cal D}$ is a definable subcategory of ${\rm Mod}\mbox{-}R$; then 
$${\cal D}^S_R = {\cal D} \,\cap\, {\rm DefTr}^S_R$$
and ${\cal D}^S = ({\cal D}^S_R)^S$.
\end{lemma}

A definable subcategory of ${\rm Mod}\mbox{-}S$ of the form ${\cal D}^S$ for some definable ${\cal D} \subseteq {\rm Mod}\mbox{-}R$ will be said to be {\bf directly induced} (from (${\rm Mod}\mbox{-}$)$R$ along $f$).  By \ref{deftrextn} every definable subcategory of ${\rm Mod}\mbox{-}S$ directly induced along $f$ has the form ${\cal D}^S$ for some definable subcategory ${\cal D}$ of ${\rm Mod}\mbox{-}R$ contained in ${\rm DefTr}^S_R$.

\begin{prop}\label{lattdefext} \marginpar{lattdefext} Suppose that $R \to S$ is a ring homomorphism.  The map $(-)^S:{\cal D} \mapsto {\cal D}^S$ is a morphism from the lattice ${\rm DefSub}_R$ of definable subcategories of ${\rm Mod}\mbox{-}R$ to ${\rm DefSub}_S$.

The lattice of definable subcategories of ${\rm Mod}\mbox{-}S$ directly induced from $R$ is a sublattice of ${\rm DefSub}_S$.  It is a complete lattice and its inclusion into ${\rm DefSub}_S$ preserves infinite meets.

The map, $(-)^S:{\cal D} \mapsto {\cal D}^S$, restricted to the definable subcategories of ${\rm DefTr}^S_R$ and corestricted to the directly induced (from $R$) definable subcategories of ${\rm Mod}\mbox{-}S$ is a lattice isomorphism.
\end{prop}
\begin{proof} If $({\cal D}_\lambda)_\lambda$ are definable subcategories of ${\rm Mod}\mbox{-}R$, then their intersection ${\cal D}$ is definable (\ref{infmeetdef}) and hence is the meet of these in the lattice of definable subcategories of ${\rm Mod}\mbox{-}R$.  We have $M \in \bigcap_\lambda \, ({\cal D}_\lambda)^S$ iff $M_R \in {\cal D}_\lambda$ for each $\lambda$, iff $M_R\in {\cal D}$ iff $M\in {\cal D}^S = \big(\bigcap_\lambda \, {\cal D}_\lambda)^S$, so $(-)^S$ commutes with arbitrary intersections.

Take two directly induced definable subcategories of ${\rm Mod}\mbox{-}S$, by \ref{deftrextn} without loss of generality of the form $({\cal D}_1)^S$ and $({\cal D}_2)^S$ for some definable subcategories ${\cal D}_1$, ${\cal D}_2$ contained in ${\rm DefTr}^S_R$.  Let $N_i$ be an elementary cogenerator for $({\cal D}_i)^S$.  Since, by \ref{deftrextn}, we have $({\cal D}_i)^S|_R = {\cal D}_i$, it follows by \ref{restelemcog} that $(N_i)_R$ is an elementary cogenerator for ${\cal D}_i$ so, by \ref{elemcogjoin}, also $(N_1)_R \oplus (N_2)_R = (N_1 \oplus N_2)_R$ is an elementary cogenerator for the join ${\cal D}_1 \vee {\cal D}_2$.  Since $N_1 \oplus N_2$ is an elementary cogenerator for $({\cal D}_1)^S \, \vee \, ({\cal D}_2)^S$, we see that $(-)^S$ preserves finite joins.

If ${\cal D}_1 \subseteq {\cal D}_2 \subseteq {\rm DefTr}^S_R$ have the same image under $(-)^S$ and if $M_2 \in {\cal D}_2$, then, since ${\cal D}_2 \subseteq {\rm DefTr}^S_R$, there is, by \ref{restr}, $M_S \in ({\cal D}_2)^S$ with $M_2$ pure in $M_R$.  By assumption $M\in ({\cal D}_1)^S$ so $M_2$, being pure in $M_R$, is in ${\cal D}_1$, showing that $(-)^S$ is injective on definable subcategories of ${\rm DefTr}^S_R$.
\end{proof}

The inclusion does not preserve infinite joins, as is shown by the next example.

\begin{example} \label{infjoin} \marginpar{infjoin} Take $R=k[x]$, $S=k[x,y,z: y(y-x)=0=y(yz-1)]$, where $k$ is a field, let $f:R \to S$ be the natural inclusion.  Let $M_n$ be the $R$-module $k[x]/(x^n)$.

Suppose that $M$ is an $S$-module whose restriction to $R$ is isomorphic to $M_n$, say $M$ is generated by $a$ which satisfies $ax^n=0$.  We show that $ay=0$.  From the relations on $S$ we have $ay =ay^2z = ay.yz = ay^3z^2 = \dots = ay^{n+1}z^n$.  Also $ay^2= ayx$, so $ay^{n+1} = axy^n = \dots = ax^ny =0$.  So $ay=0$.  

Let ${\cal D}_n$ be the definable subcategory of ${\rm Mod}\mbox{-}R$ generated by $k[x]/(x^n)$ and suppose that $M\in ({\cal D}_n)^S$, that is $M_R \in {\cal D}_n$.  Since $x^n$ acts as $0$ on $k[x]/(x^n)$, the same is true of every module in ${\cal D}_n$, in particular on $M_R$.  So the computation above shows that $My=0$.

Let ${\cal D}$ be the join of the definable subcategories ${\cal D}_n$.  Then \cite[5.1]{Zie} ${\cal D}$ contains the module $k(x)_{k[x]}$, where $k(x)$ is the ring of rational functions.  Note that $x$ acts invertibly on this module.  We may make $k(x)$ into an $S$-module by defining the action of $y$ to agree with that of $x$ (and $z$ to act as $x^{-1}$).  Thus we obtain a module in ${\cal D}^S$ on which the action of $y$ is nonzero.  

On the other hand, since the pp-pair $v=v/vy=0$ is closed on each $({\cal D}_n)^S$ it is also closed on the join of these definable subcategories in ${\rm DefSub}_S$ which, therefore, does not coincide with ${\cal D}^S$.
\end{example}

By \ref{lattdefext}, if ${\cal C}$ is a definable subcategory of ${\rm Mod}\mbox{-}S$, then there is a smallest definable subcategory of ${\rm Mod}\mbox{-}S$ directly induced from $R$ and containing ${\cal C}$.  We refer to this as the $R${\bf -closure} of ${\cal C}$ (with respect to the given ring morphism $R \to S$).

\begin{lemma} \label{restrind} \marginpar{restrind} If ${\cal C}$ is a definable subcategory of ${\rm Mod}\mbox{-}S$, then the definable subcategory, $({\cal C}|_R)^S$ of ${\rm Mod}\mbox{-}S$ directly induced from the restriction ${\cal C}|_R$ is the $R$-closure of ${\cal C}$.  The $R$-closure of ${\cal C}$ is defined by the set of those pp-pairs which are closed on ${\cal C}$ and which are of the form $f_\ast\phi/f_\ast \psi$ for some $\phi, \psi$ in the language of $R$-modules.  The definable subcategory of ${\rm Mod}\mbox{-}R$ defined by the set of corresponding pp-pairs $\phi/\psi$ is ${\cal C}|_R$.
\end{lemma}
\begin{proof} The first statement is clear from the definitions, the second and third statements follow directly from the definitions, \ref{fstarclosed} and \ref{ppdirextn}.
\end{proof}

It follows from \ref{Zgdef} and \ref{lattdefext} that direct extension is an embedding at the level of topology, that is, a frame embedding, from the Ziegler spectrum of $R$ to that of $S$.

\begin{cor}\label{dirextfrm} \marginpar{dirextfrm}  If $f:R \to S$ is a morphism of rings, then direct extension of definable subcategories gives a morphism of frames ${\cal O}({\rm Zg}_R) \to {\cal O}({\rm Zg}_S)$ which, restricted to the subspace ${\cal O}({\rm Zg}({\rm DefTr}^S_R))$, is an embedding.
\end{cor}
\begin{proof}  Each open subset of ${\rm Zg}_R$ has the form ${\rm Zg}_R \setminus {\rm Zg}({\cal D})$ for some definable subcategory ${\cal D}$ of ${\rm Mod}\mbox{-}R$.  So the map ${\rm Zg}_R \setminus {\rm Zg}({\cal D}) \mapsto {\rm Zg}_S \setminus {\rm Zg}({\cal D}^S)$ between open subsets of ${\rm Zg}_R$ and ${\rm Zg}_S$ preserves, by \ref{lattdefext}, finite intersections and arbitrary unions of open subsets, that is, is a morphism of frames.  Also by \ref{lattdefext} that map is injective when restricted to a map on the open subsets of ${\rm Zg}({\rm DefTr}^S_R)$ where this subset is given the relative topology.
\end{proof}

\paragraph{Extension along ring epimorphisms}

We recall that if $f:R\to S$ is an epimorphism of rings, then the situation is very straightforward.  Indeed, the action of each element of $S$ is pp-definable in terms of the actions of elements of $R$, so the model theory of $S$-modules is contained in that of $R$-modules, see \cite[6.1.9]{PreNBK}.  For instance, every (pp) formula for $S$-modules is equivalent, on every $S$-module, to a (pp) formula for $R$-modules. Indeed, restriction along $f$ is a full and faithful embedding of ${\rm Mod}\mbox{-}S$ into ${\rm Mod}\mbox{-}R$ and, see \cite[5.5.4]{PreNBK}, the image of that embedding is a definable subcategory of ${\rm Mod}\mbox{-}R$; therefore ${\cal D}^S = {\cal D} \cap {\rm Mod}\mbox{-}S$ for every definable subcategory ${\cal D}$ of ${\rm Mod}\mbox{-}R$.

\section{Tensor extension of definable subcategories} \label{sectensextn} \marginpar{sectensextn}

Suppose that $f:R\to S$ is a ring homomorphism.  Given a definable subcategory ${\cal D}$ of ${\rm Mod}\mbox{-}R$, we may consider its {\bf tensor extension} to ${\rm Mod}\mbox{-}S$:
$${\cal D}\otimes_RS = \langle \{ M\otimes_RS: M\in {\cal D}\}\rangle$$
- the definable subcategory of ${\rm Mod}\mbox{-}S$ generated by what is obtained by applying $-\otimes_RS$ to ${\cal D}$.

\begin{example}\label{cextens1} \marginpar{cextens1} The class $\{ M\otimes_RS: M\in {\cal D}\}$ appearing above might not be a definable subcategory of ${\rm Mod}\mbox{-}S$.  Take $R \to S$ to be, for instance, the diagonal map $K \to K \times K$ where $K$ is a field, and take ${\cal D}$ to be ${\rm Mod}\mbox{-}K$.  Then the modules in $\{ M\otimes_RS: M\in {\cal D}\}$ all will be even-dimensional so will not include, for instance, the $K \times K$-module $(K,0)$, which clearly is a direct summand of $(K,K) = K \otimes_K(K \times K)$.
\end{example}

\begin{example}\label{cextens2} \marginpar{cextens2}  Closing the class $\{ M\otimes_RS: M\in {\cal D}\}$ under pure submodules is, in general, not enough to give a definable subcategory:  take $R \to S$ to be $K \to K[X]$ and ${\cal D}$ to be ${\rm Mod}\mbox{-}K$.  Then $\{ M\otimes_RS: M\in {\cal D}\}$ is the category of free $K[X]$-modules which is closed under neither direct products nor directed colimits (by coherence of $K[X]$, its definable closure is the category of flat $K[X]$-modules, see e.g.~\cite[3.4.24]{PreNBK}).

This example also shows that if $N$ is an elementary cogenerator for ${\cal D}$ then $N\otimes_RS$ need not be an elementary cogenerator for ${\cal D}\otimes_RS$.  For, in this example, $K\otimes_KK[X]$ is not even pure-injective and, even if we replace the ring $K[X]$ by, say, its completion $\overline{K[X]_{(X)}}$ at the ideal generated by $X$, that module, though pure-injective, is not an elementary cogenerator for the flat $\overline{K[X]_{(X)}}$-modules (the ring of quotients of $\overline{K[X]_{(X)}}$ doesn't embed as a module in any power of $\overline{K[X]_{(X)}}$).
\end{example}

\begin{example}\label{tensnotmeet} \marginpar{tensnotmeet} Tensoring up with $S$ preserves neither infinite join nor finite meet of definable categories.  

For join, take $f:R \to S$ to be the inclusion ${\mathbb Z} \to {\mathbb Q}$ and let ${\cal D}_n$ be the definable subcategory generated by ${\mathbb Z}_{p^n} = {\mathbb Z}/p^n{\mathbb Z}$.  So ${\cal D}_n$ consists of the direct sums of copies of ${\mathbb Z}_{p^n}$.  Then ${\cal D}_n \otimes_{\mathbb Z} {\mathbb Q} =0$, so $\bigvee_n \, ({\cal D}_i \otimes_{\mathbb Z} {\mathbb Q}) =0$.  On the other hand, the torsionfree module ${\mathbb Z}_{(p)}$ belongs to $\bigvee_n {\cal D}_i$ and so $(\bigvee_n {\cal D}_i) \otimes_{\mathbb Z} {\mathbb Q} \neq 0$.

For meet, take $R \to S$ to be the surjection of rings ${\mathbb Z} \to {\mathbb Z}_p$.  Let ${\cal D}_1 = \langle {\mathbb Z}\rangle = {\rm Flat}\mbox{-}{\mathbb Z}$, ${\cal D}_2 = \langle {\mathbb Z}_p \rangle$.  Then ${\cal D}_1 \cap {\cal D}_2 =0$ but ${\cal D}_1 \otimes {\mathbb Z}_p = {\cal D}_2 \otimes {\mathbb Z}_p = {\rm Mod}\mbox{-}{\mathbb Z}_p$.
\end{example}

Depending on the type of morphism $f$ and the properties of the (bi)module $_RS_S$, there may be little relation between ${\cal D}^S$ and ${\cal D}\otimes_RS$ or between, say, ${\cal D}$ and $({\cal D}\otimes_RS)|_R$.  For instance, in Example \ref{cextens2}, ${\cal D}^{K[X]} = {\rm Mod}\mbox{-}K[X]$ whereas ${\cal D} \otimes_KK[X] = {\rm Flat}\mbox{-}K[X]$.  Or take $f:R \to S$ to be ${\mathbb Z} \to {\mathbb Z}_2$ and ${\cal D}$ to be the subcategory of torsionfree abelian groups; then ${\cal D} \otimes {\mathbb Z}_2$ is ${\rm Mod}\mbox{-}{\mathbb Z}_2$, whereas ${\cal D} \cap {\rm Mod}\mbox{-}{\mathbb Z}_2$ is $0$.

Later we will consider some special cases where there is more connection, particularly that where $f:R\to S$ is an elementary extension of rings.  We will also, in Section \ref{sectensbimod}, generalise the functor $-\otimes_RS_S$ and consider the effect of a functor $-\otimes_RB_S$ where $B$ is an $(R,S)$-bimodule.  Before that, we review some results on the interaction of pp formulas and tensor product and also consider a generalisation of the Mittag-Leffler condition.

\subsection{Tensor product and pp formulas} \label{sectenspp} \marginpar{sectenspp}

First, we state Herzog's characterisation in terms of pp formulas of when a tensor product is $0$.  Recall from Section \ref{secdefsub} that if $\phi$ is a pp formula (in $n$ free variables) for right modules then its elementary dual $D\phi$ is a pp formula (in $n$ free variables) for left modules.

\begin{theorem}\label{herzcrit} \marginpar{herzcrit} (Herzog's Criterion \cite[3.2]{HerzDual}) Suppose that $\overline{a}$ is an $n$-tuple from the right $R$-module $M$ and that $\overline{b}$ is an $n$-tuple from the left $R$-module $L$.  Then $\sum_{i=1}^n\, a_i \otimes b_i$, $\overline{a} \otimes \overline{b}$ for short, is $0$ in $M\otimes_R L$ iff there is a pp formula $\phi$ (in $n$ free variables) for right modules such that $M_R\models \phi(\overline{a})$ and $_RL \models D\phi(\overline{b})$.
\end{theorem}

\begin{cor}\label{tenszero} \marginpar{tenszero} If ${\rm pp}^M(\overline{a})$ is generated by $\phi$, then $\overline{a} \otimes \overline{b}=0$ in $M\otimes_RL$ iff $_RL \models D\phi(\overline{b})$.
\end{cor}

\begin{cor}\label{reltenszero} \marginpar{reltenszero} Suppose that $\overline{a}$ is an $n$-tuple from the right $R$-module $M$ and let ${\cal D}$ be a definable subcategory of ${\rm Mod}\mbox{-}R$.  Suppose that the left $R$-module $L$ belongs to the dual definable subcategory, ${\cal D}^{\rm d}$ and let $\overline{b}$ be an $n$-tuple from $L$.  If ${\rm pp}^M(\overline{a})$ is ${\cal D}$-generated by $\phi$, then $\overline{a} \otimes \overline{b}=0$ in $M\otimes_RL$ iff $_RL \models D\phi(\overline{b})$.
\end{cor}
\begin{proof}  One direction is by \ref{herzcrit}.  For the other, if $\overline{a} \otimes \overline{b}=0$ then, by \ref{herzcrit}, we have $M\models \psi(\overline{a})$ and $L\models D\psi(\overline{b})$ for some pp $\psi$.  By assumption $\phi \leq_{\cal D} \psi$ so, by \ref{dualord}, $D\phi \geq_{{\cal D}^{\rm d}} D\psi$ and so, since $L\in {\cal D}^{\rm d}$, we have $L\models D\phi(\overline{b})$, as claimed.
\end{proof}

Note that, in the above result, we don't have to assume that $M \in {\cal D}$.

There is a more general theorem, recalled below (\ref{sigmaphi2}), which is applicable when tensoring with a bimodule $_RB_S$.  Herzog's criterion is the case in \ref{sigmaphi2} where the pp formula $\sigma$ for right $S$-modules is $y=0$.

First we need a generalisation of the dual $D\phi$ of a pp formula.  Given rings $R$ and $S$, a pp formula $\phi(\overline{x})$ for right $R$-modules and a pp formula $\sigma(\overline{y})$ for right $S$-modules, we produce, \ref{sigmaphi1} below, a pp formula $(\sigma: \phi)$ for $(R,S)$-bimodules which expresses the property, of a tuple of appropriate length, that ``tensored with a tuple satisfying $\phi$ the result will satisfy $\sigma$".  The dual $D\phi$ is the special case $(y=0\,;\, \phi)$ (with $S={\mathbb Z}$ or, indeed, $S$ arbitrary).

We give the statement, then explain what we mean by a partitioned pp formula.  We include a proof since, in the original reference, in order to avoid notational complexity, the proof was given only in the case where $\sigma$ has one free variable.  Here we take the opportunity to give the proof of the general case.

\begin{theorem}\label{sigmaphi1} \marginpar{sigmaphi1} (\cite[2.1(1,2)]{PreTens}) Suppose that $\sigma$ is a pp formula for right $S$-modules and that $\phi$ is a (partitioned) pp formula for right $R$-modules.  Then there is a (matching partitioned) pp formula $(\sigma: \phi)$ for $(R,S)$-bimodules, equivalently for right $R^{\rm op}\otimes S$-modules, such that, for every right $R$-module $M_R$ and $(R,S)$-bimodule $_RB_S$, for every $\overline{a}$ from $M$ and matching $\overline{b}$ from $B$:

\noindent (1) if $M_R\models \phi(\overline{a})$ and $_RB_S \models (\sigma: \phi)(\overline{b})$, then $(M\otimes_RB)_S \models \sigma(\overline{a} \otimes \overline{b})$;

\noindent (2) if ${\rm pp}^M(\overline{a})$ is generated by $\phi$, then $M\otimes_RB_S\models \sigma (\overline{a} \otimes \overline{b})$ iff $_RB_S \models (\sigma: \phi)(\overline{b})$.
\end{theorem}

If the formula $\sigma$ has just one free variable, then the references to partition can be ignored.  Otherwise, say if $\sigma$ has $k$ free variables, we have to specify, as well as the formula $\phi$, the way in which the free variables of $\phi$ are partitioned into $k$ blocks.  For instance, if $k=2$ and $\phi$ has $3$ free variables, then we have to specify how we tensor two $3$-tuples to get a 2-tuple: whether $(a_1\otimes b_1, a_2\otimes b_2 + a_3\otimes b_3$) or $(a_1\otimes b_1 + a_2\otimes b_2, a_3\otimes b_3)$.  If $\zeta$ is a partition of $n$ into $k$ parts then by a $\zeta$-{\bf partitioned} formula we mean one with $n$ free variables and whose free variables $x_1, \dots, x_n$ are partitioned according to $\zeta$.  A {\bf matching} formula is another $\zeta$-partitioned formula.  We conveniently regard such a formula as having free variables presented as a $k$-tuple $(\overline{x}_1, \dots, \overline{x_k})$ of tuples $\overline{x}_t$ of variables.  Usually we don't have to specify more than this (such as the length of each part) and we may just refer to a formula as above as being a $k$-{\bf partitioned} formula.  Of course, we assume that formulas and tuples of elements match as needed, even if that is not specified.

\vspace{4pt}

\begin{proof} The proof in \cite{PreTens} is given just for the case where the partition has a single part, that is, where $\sigma$ has just one free variable.  The general case proceeds similarly and we give that here for completeness.

Suppose that $\sigma = \sigma(v_1,\dots, v_k)$ has $k$ free variables and take $\sigma$ explicitly to be the formula
$$\exists \overline{w} \, \bigwedge_j  \sum_l \, v_ls _{lj} + \sum_i w_it_{ij} =0.$$

Next consider the partitioned formula $\phi = \phi(x_1, \dots, x_n) = \phi(\overline{x}_1, \dots, \overline{x}_k)$ and choose a {\bf free realisation} $(C_\phi, \overline{a}= (\overline{a}_1 \, \dots\, \overline{a}_k))$ of $\phi$, meaning that $C_\phi$ is finitely presented and ${\rm pp}^{C_\phi}(\overline{a}) = \langle \phi \rangle$ (such exists, \cite[1.2.14]{PreNBK}).  Take a finite generating tuple $\overline{c}$ for the $R$-module $C_\phi$ and let $\theta$ be a conjunction\footnote{finite, since $C_\phi$ is finitely presented} of $R$-linear equations which together generate all the relations on $\overline{c}$ in $C_\phi$; suppose that $\overline{c}$ is an $m$-tuple.  Also choose a matrix $H$ such that $\overline{a} = \overline{c}H$ and note that we can write $H$ as a partitioned matrix $(H_1 \, \dots \, H_k)$ so that $\overline{a}_l = \overline{c}H_l$.

We set the pp formula (for $(R,S)$-bimodules) $(\sigma: \phi)(\overline{y})$, with free variables $\overline{y} = (\overline{y}_1 \,\dots\, \overline{y}_k)$ an $n$-tuple partitioned as is $\overline{x}$, to be 
$$\exists \overline{u}\, \bigwedge_j \, D\theta(\sum_l H_l \overline{y}_l \cdot s_{lj} + \sum_i \overline{u}_i\cdot t_{ij})$$ 
where $\overline{u}$ is a tuple (of length equal to that of $\overline{w}$) of $m$-tuples.

\vspace{4pt}

\noindent Proof of (2)($\Leftarrow$) and (1):  With notation as in the statement of (2), write $\overline{b}$ from $_RB_S$ as $(\overline{b}_1 \,\dots\, \overline{b}_k)$ to match the partitioned $n$-tuple $\overline{a}$.  

Suppose first that $_RB_S \models (\sigma: \phi)(\overline{b})$.  So there is $\overline{d}$, a tuple of $m$-tuples $\overline{d}_i$ from $B$, such that 

$$\bigwedge_j \, D\theta(\sum_l H_l \overline{b}_l \cdot s_{lj} + \sum_i \overline{d}_i\cdot t_{ij}).$$

Also, since $M\models \phi(\overline{a})$, then there is $\overline{c}$ from $M$ such that $\overline{a} = \overline{c}H$ and such that $M \models \theta(\overline{c})$.
Then, for each $j$, we have, since $M \models \theta(\overline{c})$, 
$$\sum_l \overline{c} \otimes H_l \overline{b}_l \cdot s_{lj} + \sum_i \overline{c} \otimes \overline{d}_i\cdot t_{ij}=0,$$
(regard tuples occurring on the right of tensors as column vectors) 

\noindent from which we obtain

$$\sum_l \overline{c} H_l \otimes \overline{b}_l \cdot s_{lj} + \sum_i \overline{c} \otimes \overline{d}_i\cdot t_{ij}=0,$$
that is
$$\sum_l \overline{a}_l \otimes \overline{b}_l \cdot s_{lj} + \sum_i \overline{c} \otimes \overline{d}_i\cdot t_{ij}=0.$$

Thus $M \otimes_RB_S \models \sigma (\overline{a} \otimes \overline{b})$, as required.

\vspace{4pt}

For the converse, 2)($\Rightarrow$), suppose that $C_\phi \otimes_RB_S \models \sigma(\overline{a} \otimes \overline{b})$, that is $C_\phi \otimes_RB_S \models \sigma(\overline{a}_1 \otimes \overline{b}_1, \dots , \overline{a}_k \otimes \overline{b}_k)$.  So there are $\overline{m}_i$ from $C_\phi$ and $\overline{d}_i$ from $B$ such that
$$\bigwedge_j \, \sum_l \overline{a}_l\otimes \overline{b}_l.s_{lj} + \sum_i \overline{m}_i \otimes \overline{d}_i.t_{ij} =0.$$
Since $C_\phi$ is generated by $\overline{c}$, we have, for each $i$, $\overline{m}_i = \overline{c}K_i$ for some matrix $K_i$ with entries in $R$.  Therefore $\overline{m}_i\otimes \overline{d}_it_{ij} = \overline{c}K_i \otimes \overline{d}_it_{ij} = \overline{c}\otimes K_i\overline{d}_it_{ij}$.  Writing $\overline{e}_i = K_i\overline{d}_i$ and recalling that $\overline{a}_l = \overline{c}H_l$, this gives us
$$\bigwedge_j \, \sum_l \overline{c}_l\otimes H_l\overline{b}_l.s_{lj} + \sum_i \overline{c} \otimes \overline{e}_i.t_{ij} =0,$$
that is,
$$\bigwedge_j \, \overline{c} \otimes (\sum_l H_l\overline{b}_l.s_{lj} + \sum_i  \overline{e}_i.t_{ij}) =0.$$
By Herzog's criterion, \ref{herzcrit}, and since ${\rm pp}^{C_\phi}(\overline{c}) = \langle \theta \rangle$, it follows that, for each $j$, we have
$_RB_S \models D\theta(\sum_l H_l\overline{b}_l.s_{lj} + \sum_i  \overline{e}_i.t_{ij})$ and hence that 
$$_RB_S \models \bigwedge_j \, D\theta(\sum_l H_l\overline{b}_l.s_{lj} + \sum_i  \overline{e}_i.t_{ij}).$$
That is
$$_RB_S \models \exists \overline{u} \,\bigwedge_j \, D\theta(\sum_l H_l\overline{b}_l.s_{lj} + \sum_i  \overline{u}_i.t_{ij}),$$
as required.

The proof for the general case, with $M$ in place of $C_\phi$, is at the end of the proof of \ref{sigmaphi2} below.
\end{proof}

The next result is the generalisation of \ref{herzcrit}.

\begin{theorem}\label{sigmaphi2} \marginpar{sigmaphi2} (\cite[2.1(3)]{PreTens}) Suppose that $M_R$ is a right $R$-module, $_RB_S$ is an $(R,S)$-bimodule, $\sigma$ is a pp formula for right $S$-modules, $\overline{a}$ is a partitioned tuple from $M$ and $\overline{b}$ is a matching partitioned tuple from $B$.  Then $(M\otimes_RB)_S \models \sigma(\overline{a} \otimes \overline{b})$ iff there is a (matching partitioned) pp formula $\phi$ for right $R$-modules such that $M_R \models \phi(\overline{a})$ and $_RB_S \models (\sigma: \phi)(\overline{b})$.
\end{theorem}
\begin{proof} The proof is just as in \cite{PreTens} and we repeat it here.  We already have one direction by \ref{sigmaphi1}, so suppose that $M\otimes_RB_S \models \sigma(\overline{a} \otimes \overline{b})$.  With notation as in the above proof, that means we have $\overline{m}_i$ from $M$ and $\overline{d}_i$ from $B$ such that 
$$\bigwedge_j \, \sum_l \overline{a}_l\otimes \overline{b}_l.s_{lj} + \sum_i \overline{m}_i \otimes \overline{d}_i.t_{ij} =0.$$
Then, by \ref{herzcrit}, for each $j$, there is a pp formula $\phi_j$ for $R$-modules such that 
$$M\models \phi_j(\overline{a}, \overline{m}_1, \dots, \overline{m}_s),$$ 
where $\overline{a} = (\overline{a}_1 \, \dots \, \overline{a}_k)$ and $i$ runs from 1 to $s$,
and
$$_RB_S \models D\phi_j(\overline{b}, \overline{d}_1, \dots, \overline{d}_s).$$

Let $\phi(\overline{x})$ be the pp formula $\exists \overline{u}_1, \dots, \overline{u}_s \, \bigwedge_j \phi_j(\overline{x}, \overline{u}_1, \dots, \overline{u}_s)$, so $M \models \phi(\overline{a})$.  We claim that $_RB_S \models (\sigma: \phi)(\overline{b})$.  By \ref{sigmaphi1}(2), it is sufficient to show that if $(C_\phi, \overline{a}')$, with $\overline{a}' = (\overline{a}'_1 \, \dots \, \overline{a}'_k)$, is a free realisation of $\phi$, then $C_\phi \otimes B \models \sigma(\overline{a}'\otimes \overline{b})$.  Since $C_\phi \models \phi(\overline{a}')$, there are $\overline{c}_1, \dots, \overline{c}_s$ such that $C_\phi \models \phi_j(\overline{a}', \overline{c}_1, \dots, \overline{c}_s)$ for each $j$.  Therefore 
$$C_\phi \otimes B \models \sum_l \overline{a}'_l\otimes \overline{b}_l.s_{lj} + \sum_i \overline{c}_i \otimes \overline{d}_i.t_{ij} =0$$ 
for each $j$
and hence $C_\phi \otimes_R B_S \models \sigma(\overline{a}' \otimes \overline{b})$, as required.

To finish the proof of \ref{sigmaphi1}, if $M\otimes B\models \sigma(\overline{a} \otimes \overline{b})$, so $M\models \psi(\overline{a})$ and $B \models (\sigma: \psi)(\overline{b})$ for some $\psi$ then, since, by assumption $\phi \leq \psi$, we have $(\sigma: \psi) \leq (\sigma: \phi)$ and so $B\models (\sigma: \phi)(\overline{b})$, as required.
\end{proof}

Following easily from the definitions, we have some elementary properties of $(-:-)$.

\begin{lemma}\label{sigmaphiprop} \marginpar{sigmaphiprop} (\cite[\S 2.2 (1), (2)]{PreTens}) 

\noindent (1) $(\sigma: \phi + \psi) = (\sigma: \phi) \wedge (\sigma: \psi)$, in particular, if $\phi \leq \psi$, then $(\sigma: \psi) \leq (\sigma: \phi)$.

\noindent (2) $(\sigma: \phi \wedge \psi) \geq (\sigma: \phi) + (\sigma: \psi)$.
\end{lemma}

\begin{remark} \label{fmlaoverR} \marginpar{fmlaoverR} We will often apply this with $_RB_S =\, _RS_S$ where $R\to S$ is a morphism of rings.  Note that, if $\sigma$ is a pp formula for right $R$-modules, then $(\sigma: \phi)$ will be a formula for $(R,R)$-bimodules:  this can be seen from the proof of \ref{sigmaphi1} but also follows from the statement since, if we read the result with $_RS_R$ in place of $_RS_S$, then nothing changes about the properties of $(\sigma: \phi)$, hence that formula is unchanged (up to equivalence).
\end{remark}

\begin{cor} \label{elem4} \marginpar{elem4} Suppose that $f:R \to S$ is any morphism of rings, that $M$ is an $R$-module and that $\overline{a} = (a_1, \dots, a_n)$ is a tuple from $M$.  Consider $\overline{a} \otimes \overline{1} = (a_1 \otimes 1_S, \dots, a_n \otimes 1_S)  \in M\otimes_RS_S$.

Then
$${\rm pp}^{M\otimes S}(\overline{a}\otimes \overline{1}) = \{ \sigma: \text{there is } \phi \in {\rm pp}^M(\overline{a}) \text{ with }(\sigma: \phi) \in {\rm pp}^{_RS_S}(\overline{1})\}$$
and ${\rm pp}^{M\otimes S}(\overline{a}\otimes \overline{1})$ is generated by $f_\ast {\rm pp}^M(\overline{a})$.

In particular, if ${\rm pp}^M(\overline{a})$ is finitely generated (by $\phi$) then so is ${\rm pp}^{M\otimes S}(\overline{a}\otimes \overline{1})$ (by $f_\ast\phi$).
\end{cor}
\begin{proof}  The first part is immediate from \ref{sigmaphi2}.

Since $M_R \to M\otimes_RS_R$ given by $a \mapsto a\otimes 1_S$ is an $R$-linear map ($ar\otimes 1 = a\otimes fr$) we have $f_\ast\phi \in {\rm pp}^{M\otimes S}(\overline{a} \otimes \overline{1})$ for every $\phi \in {\rm pp}^M(\overline{a})$.  For the converse, if $\sigma \in {\rm pp}^{M\otimes S}(\overline{a} \otimes \overline{1})$, then for some $\phi \in {\rm pp}^M(\overline{a})$ we have $(\sigma: \phi) \in {\rm pp}^{_RS_S}(\overline{1})$; we claim that $f_\ast\phi \leq \sigma $.  So suppose that $N$ is an $S$-module and that $N\models f_\ast\phi(\overline{c})$.  Then (\ref{fstar}) $N_R \models \phi(\overline{c})$ and so, since $_{R}S_S \models (\sigma: \phi)(\overline{1})$, we have $N \otimes_{R} S_S \models \sigma(\overline{c} \otimes \overline{1})$.  We have the $S$-linear map $N \otimes_{R} S_S \to N\otimes_S S_S \simeq N$, so it follows that $N \models \sigma(\overline{c})$.  Therefore $f_\ast\phi \leq \sigma $, as claimed.  The statement about finite generation is then immediate.
\end{proof}

\subsection{Mittag-Leffler modules} \label{secML} \marginpar{secML}

Recall that a module $M$ is {\bf Mittag-Leffler} \cite{RayGru} if the equivalent conditions of the following theorem are fulfilled.

\begin{theorem} \label{MLchar} \marginpar{MLchar} (\cite[2.2]{RotHab}, see \cite{RotML}) Suppose that $M$ is a right $R$-module.  Then the following conditions are equivalent.

\noindent (i) For every set $\{ L_i\}_{i\in I}$ of left $R$-modules, the canonical map
$$M\otimes_R \, (\prod_{i\in I} \, L_i) \to \prod_{i\in I} \, (M\otimes_R L_i)$$ 
is monic.

\noindent (ii)  The pp-type ${\rm pp}^M(\overline{a})$ of any finite tuple in $M$ is finitely generated.
\end{theorem}

For example, finitely presented modules, more generally {\bf pure-projective} modules (direct summands of direct sums of finitely presented modules), are Mittag-Leffler.

\vspace{4pt}

There are relative versions of these definitions and results, \cite{RotHab}, see \cite[\S 3]{RotML2}, as follows.

If ${\cal K}$ is a class of left $R$-modules then a right $R$-module is ${\cal K}$-{\bf Mittag-Leffler} if, for every set $\{ L_i\}_{i\in I}$ of modules in ${\cal K}$, the canonical map $M\otimes_R \, (\prod_{i\in I} \, L_i) \to \prod_{i\in I} \, (M\otimes_R L_i)$ is monic.

\begin{lemma}\label{defML} \marginpar{defML} (see \cite[3.6]{RotML2}) If ${\cal K}$ is a class of left $R$-modules and $\langle {\cal K} \rangle$ is the definable subcategory of $R\mbox{-}{\rm Mod}$ that it generates, then a right module $M$ is ${\cal K}$-Mittag-Leffler iff it is $\langle {\cal K} \rangle$-Mittag-Leffler.
\end{lemma}

Then, generalising \ref{MLchar} we have the following, where, given a definable subcategory ${\cal D}$ of ${\rm Mod}\mbox{-}R$, we say that $M_R$, a module not necessarily in ${\cal D}$, is ${\cal D}$-{\bf atomic} if the pp-type ${\rm pp}^M(\overline{a})$ of any finite tuple in $M$ is ${\cal D}$-finitely generated (definition just after \ref{ppclosed}).

\begin{theorem}\label{relMLchar} \marginpar{relMLchar} \cite[2.2]{RotHab} (see \cite[3.1, 3.6]{RotML2})  Suppose that ${\cal K}$ is a class of left $R$-modules and $M$ is a right $R$-module.  The following conditions on $M$ are equivalent:

\noindent (i) $M$ is ${\cal K}$-Mittag-Leffler;

\noindent (ii) $M$ is ${\cal D}$-{\bf atomic} where ${\cal D} = \langle{\cal K}\rangle^{\rm d}$ is the definable subcategory of ${\rm Mod}\mbox{-}R$ dual to the definable subcategory $\langle{\cal K}\rangle$ of $R\mbox{-}{\rm Mod}$ generated by ${\cal K}$.

If $\langle{\cal K} \rangle = \langle L \rangle$, then another equivalent is:

\noindent (iii) for every finite tuple $\overline{a}$ from $M$, there is $\phi \in {\rm pp}^M(\overline{a})$ such that the tensor-annihilator, $\{ \overline{b} \text{ from } L: \overline{a} \otimes \overline{b} =0\}$, of $\overline{a}$ in $_RL$ is exactly $D\phi(L)$.
\end{theorem}
\begin{proof} The equivalence of (i) and (ii) is in the references but the equivalence of (ii) and (iii) is not explicitly there.  (ii)$\Rightarrow$(iii) is by \ref{reltenszero}.  For the other direction, we have, by \ref{herzcrit}, that the tensor-annihilator of $\overline{a}$ in $L$ is $\bigcup_{\psi \in {\rm pp}^M(\overline{a})} \, D\psi(L)$.  Therefore, for $\psi\in {\rm pp}^M(\overline{a})$, we have $D\psi(L) \leq D\phi(L)$, that is, $D\psi \leq_{\langle L \rangle} D\phi$, equally $D\psi \leq_{\langle {\cal K} \rangle} D\phi$.  Therefore, by \ref{dualord}, $\psi \geq_{\cal D} \phi$, as required. 
\end{proof}

\subsection{Mittag-Leffler modules with respect to a bimodule}

First we note that condition (i) of \ref{MLchar} may be stated more strongly.

\begin{lemma}\label{corML2} \marginpar{corML2} If $M_R$ is a Mittag-Leffler module and, for each $\lambda$, $L_\lambda$ is an $(R,S)$-bimodule, then the canonical map $M \otimes_R \prod_\lambda \, L_\lambda \to \prod_\lambda \, M\otimes_R L_\lambda$ is a pure embedding of right $S$-modules.

In particular, $M_R$ is Mittag-Leffler iff for every set $\{ L_i\}_{i\in I}$ of left $R$-modules, the canonical map $M\otimes_R \, (\prod_{i\in I} \, L_i) \to \prod_{i\in I} \, (M\otimes_R L_i)$ is a pure embedding of abelian groups.
\end{lemma}
\begin{proof}   In order to prove $S$-purity of the map it is enough to show that tensoring it with any finitely presented left $S$-module $F$ gives an embedding.  The tensored map is 
$$\big(M \otimes_R \prod_\lambda \, L_\lambda)\otimes_SF \to \big(\prod_\lambda \, M\otimes_R L_\lambda)\otimes_SF.$$
Since $_SF$ is finitely presented, $-\otimes_SF$ commutes with direct products, \cite{Len} - see \ref{Lenfp}, so, on the left, we have 
$$\big(M \otimes_R \prod_\lambda \, L_\lambda)\otimes_SF \simeq M \otimes_R \big(\prod_\lambda \, L_\lambda)\otimes_SF \simeq  M\otimes_R \prod_\lambda \, (L_\lambda\otimes_SF)$$
and, on the right, we have
$$\big(\prod_\lambda \, M\otimes_R L_\lambda)\otimes_SF \simeq \prod_\lambda \, M\otimes_R L_\lambda \otimes_SF.$$
The map between these two modules is monic by applying the Mittag-Leffler property of $M$ with the left $R$-modules $_R L_\lambda \otimes_SF$, as required.
\end{proof}

\begin{definition}  Suppose that $M_R$ is a right $R$-module and that $_RB_S$ is an $(R,S)$-bimodule.  Say that $M_R$ is {\bf Mittag-Leffler} (we will write {\bf ML} for short) {\bf with respect to} \footnote{In the first version of this paper, we called this being {\em atomic} with respect to the bimodule, but that clashes with already-existing use of ``atomic", hence the change in terminology}the bimodule $_RB_S$ if, for every $k$-partitioned finite tuple $\overline{a}$ from $M$, and for every pp formula $\sigma(\overline{u})$, with $k$ free variables, in the language of right $S$-modules, there is $\phi \in {\rm pp}^M(\overline{a})$ such that $(\sigma: \phi)(B)$ is maximal (and hence, by \ref{sigmaphiprop}(2), the unique maximum) among pp-definable subgroups of $_RB_S$ of the form $(\sigma: \psi)(B)$ with $\psi \in {\rm pp}^M(\overline{a})$.
\end{definition}

This definition might seem somewhat technical, though \ref{MLpure} shows its naturality, so we explain how it generalises both the original and the relative conditions seen in Section \ref{secML}.

Consider the case where $\sigma$ is the pp formula $u=0$ and consider the above condition on $M_R$ and $_RB_S$ but required only for this formula.  Note first that the bimodule formula $(\sigma: \phi)$ then becomes (see Section \ref{sectenspp}) $D\phi$ - the elementary dual formula (usually regarded as a formula for left $R$-modules but equally a pp formula for $(R,{\mathbb Z})$-bimodules).  Then, bearing in mind Herzog's criterion, \ref{herzcrit}, the above becomes the requirement that, for every finite tuple $\overline{a}$ from $M$, there is $\phi \in {\rm pp}^M(\overline{a})$ such that the tensor-annihilator $\{ \overline{b} \in B: \overline{a} \otimes \overline{b} =0\}$ of $\overline{a}$ in $_RB$ is exactly $D\phi(_RB)$.  By \ref{relMLchar} that is exactly the condition that $M$ is Mittag-Leffler with respect to (the definable subcategory of $R\mbox{-}{\rm Mod}$ generated by) the left module $_RB$ (equally, in our terminology, with respect to the bimodule $_RB_{\mathbb Z}$).  Equivalently, by \ref{relMLchar}, $M$ is atomic with respect to the dual of that definable subcategory.

Therefore, by \ref{relMLchar}, the condition of $M$ being ML with respect to $_RB_S$, restricted just to the pp formula $\sigma$ being $u=0$, is equivalent to the condition that, for all $L_\lambda \in \langle _RB \rangle = \langle _RB_{\mathbb Z} \rangle$, the canonical morphism $M \otimes_R \prod_\lambda \, L_\lambda \to \prod_\lambda \, M\otimes_R L_\lambda$ is monic.  We will show, \ref{MLpure}, that the full condition is equivalent to the condition that this morphism be a pure embedding of right $S$-modules.  Before showing that, we note the following (which extends \ref{defML}).

\begin{lemma}\label{relatgen} \marginpar{relatgen} If $R$ and $S$ are $K$-algebras for some commutative ring $K$ and if $M_R$ is Mittag-Leffler with respect to $_RB_S$, where $K$ acts centrally, then $M_R$ is Mittag-Leffler with respect to every bimodule $_RL_S$ in the definable subcategory of $R\otimes_KS^{\rm op}\mbox{-}{\rm Mod}$ generated by $B$.

More precisely, for each pp formula $\sigma$ for right $S$-modules, and for each partitioned tuple $\overline{a}$ from $M$, if $(\sigma:\phi)(B) \geq (\sigma:\psi)(B)$ for every $\psi \in {\rm pp}^M(\overline{a})$, then the same is true for each $L\in \langle _RB_S \rangle$.
\end{lemma}
\begin{proof} Given a pp formula $\sigma$ for right $S$-modules and a tuple $\overline{a}$ from $M$, let $\phi \in {\rm pp}^M(\overline{a})$ be chosen for $\sigma$ as in the definition (of ML with respect to $_RB_S$).  That is, for every $\psi \in {\rm pp}^M(\overline{a})$, we have $(\sigma:\psi)(_RB_S) \leq (\sigma: \phi)(_RB_S)$ and therefore, by \ref{ordXeqordD}, this inclusion also holds for every $L \in \langle _RB_S  \rangle$, so $M$ is indeed Mittag-Leffler with respect to $L$.
\end{proof}

\begin{theorem} \label{MLpure} \marginpar{MLpure} Suppose that $M_R$ is a right $R$-module and that $_RB_S$ is an $(R,S)$-bimodule.  Then $M_R$ is Mittag-Leffler with respect to $_RB_S$ iff, for all $L_\lambda \in \langle _RB_S \rangle$, the canonical morphism $M \otimes_R \prod_\lambda \, L_\lambda \to \prod_\lambda \, M\otimes_R L_\lambda$ is a pure embedding of right $S$-modules.
\end{theorem}
\begin{proof} The proof is just as for the case where $\sigma$ is $u=0$.

($\Rightarrow$) Just to simplify notation, we recall that, to check purity, it is enough to consider only pp formulas $\sigma$ with one free variable, see \cite[2.1.6]{PreNBK}.  We have already noted that the hypothesis implies that $\iota: M \otimes_R \prod_\lambda \, L_\lambda \to \prod_\lambda \, M\otimes_R L_\lambda$ is monic, so let $\overline{a}$ be from $M$ and $\overline{l} = (\overline{l}_\lambda)_\lambda$ be from $\prod_\lambda \, L_\lambda$ and suppose that $\prod_\lambda \, M\otimes_R L_\lambda \models \sigma(\iota(\overline{a} \otimes \overline{l}))$, that is $\prod_\lambda \, M\otimes_R L_\lambda \models \sigma((\overline{a} \otimes \overline{l}_\lambda)_\lambda)$ and so $M\otimes_R L_\lambda \models \sigma(\overline{a} \otimes \overline{l}_\lambda)$ for each $\lambda$.  By \ref{sigmaphi2} there is, for each $\lambda$, some $\psi_\lambda \in {\rm pp}^M(\overline{a})$ such that $L_\lambda \models (\sigma: \psi_\lambda)(\overline{l}_\lambda)$.  We need a uniform choice of $\psi_\lambda$.  We have that because $M$ is ML with respect to $B$, and hence there is $\phi \in {\rm pp}^M(\overline{a})$ such that $(\sigma:\phi)(B) \geq (\sigma: \psi)(B)$ for each $\psi \in {\rm pp}^M(\overline{a})$.  By \ref{relatgen}, we have  $(\sigma:\phi)(L) \geq (\sigma: \psi_\lambda)(L_\lambda)$ and so $L_\lambda \models (\sigma: \phi)(\overline{l}_\lambda)$ for each $\lambda$.  Therefore $\prod_\lambda \, L_\lambda \models (\sigma: \phi)(\overline{l})$ and hence, by \ref{sigmaphi1}, $M\otimes_R \prod_\lambda \, L_\lambda \models \sigma(\overline{a} \otimes \overline{l})$, as required.

($\Leftarrow$)  The argument reverses: given a pp formula $\sigma$ for right $S$-modules and a partitioned tuple $\overline{a}$ from $M$, let $I$ be the set of pairs $(\psi, \overline{l}')$ where $\psi \in {\rm pp}^M(\overline{a})$ and $\overline{l}' \in (\sigma: \psi)(_RB_S)$.  Form the product $B^I$ and define $\overline{l} \in B^I$ to be the element whose coordinate at $(\psi, \overline{l}')$ is $\overline{l}'$.  Then, by \ref{sigmaphi1}, at each coordinate $(\psi, \overline{l}')$ we have $M\otimes_RB_{(\psi, \overline{l}')} \models \sigma(\overline{a} \otimes \overline{l}')$.  Therefore $M\otimes_R B^I \models \sigma(\overline{a} \otimes \overline{l})$ and hence, by \ref{sigmaphi2}, there is $\phi \in {\rm pp}^M(\overline{a})$ such that $\overline{l} \in (\sigma: \phi)(B^I)$.  Therefore, looking at each coordinate, we see $\overline{l}' \in (\sigma: \phi)(B)$.  So we have shown $(\sigma: \phi)(B) \geq (\sigma: \psi)(B)$ for every $\psi \in {\rm pp}^M(\overline{a})$, and $M$ is Mittag-Leffler with respect to $B$, as claimed.
\end{proof}

\begin{remark} Any Mittag-Leffler module is Mittag-Leffler with respect to any bimodule; this is immediate from \ref{MLpure} or just from the definitions and \ref{MLchar}.. 
\end{remark}

\begin{cor} If the bimodule $_RB_S$ has the ascending chain condition on pp-definable subgroups, then every right $R$-module $M$ is Mittag-Leffler with respect to $B$ and hence the canonical morphism $M\otimes_R B^I \to (M\otimes_RB)^I$ is a pure embedding of $S$-modules for each $I$, as is $M \otimes_R \prod_\lambda \, L_\lambda \to \prod_\lambda \, M\otimes_R L_\lambda$ if each $L_\lambda \in \langle _RB_S \rangle$.
\end{cor}
\begin{proof}  It was already noted (parenthetically) in the definition at the start of this section that, by \ref{sigmaphiprop}(2), the set of $(\sigma:\phi)$ with $\phi$ in some pp-type, is upwards-directed (since pp-types are closed under $\wedge$, hence downwards-directed).  So this is immediate from the definition.
\end{proof}

\begin{cor}\label{corML1} \marginpar{corML1} Given an $(R,S)$-bimodule $B$, the class of right $R$-modules which are Mittag-Leffler with respect to $B$ is closed under arbitrary direct sums and pure submodules.
\end{cor}

That easily follows from the definition but is also immediate from \ref{MLpure}.

\begin{remark}\label{nontriv} \marginpar{nontriv} We should show that $M_R$ being Mittag-Leffler with respect to a bimodule $_RB_S$ is genuinely stronger than requiring that it be Mittag-Leffler with respect to the left module $_RB$.  Looking at the proof of \ref{corML2}, it is equivalent to say that a module $M_R$ which is Mittag-Leffler with respect to $_RB$ need not be Mittag-Leffler with respect to every module $_RB\otimes_S F$ where $_SF$ is a finitely presented left $S$-module.

Let $R=K[T]$, let $M_R= K[T,T^{-1}]$ considered as a right $R$-module and let $_RB$ be $K\oplus K[T,T^{-1}]$ considered as a left $R$-module where $T$ has the trivial action on the first direct summand.  We claim that $M$ is $_RB$-ML.  Any power of $B$ is a power of $K[T,T^{-1}]$ direct sum a power of $K$, so checking that the canonical map $M\otimes B^I \to (M\otimes B)^I$ is monic splits into checking that $K[T,T^{-1}]$ is $K[T,T^{-1}])^I$-ML, which is the case since these are ML modules over the ring $K[T,T^{-1}]$ and that $K[T,T^{-1}]$ is ML with respect to the trivial module $K$, which is clear.

Now let $s$ be the endomorphism of $_RB$ which embeds the direct summand $K$ into $K[T,T^{-1}]$ by taking a generator $c$ to $1\in K[T,T^{-1}]$ and which sends the copy of $K[T,T^{-1}]$ to $0$; note that this does commute with the $T$-action.  Let $S=K[T,s]$ acting as $B_S$.  Let $a = 1\in M$, let $\phi_n$ be the $R$-formula $T^n|x$ and let $\sigma(y)$ be the $S$-formula $s|y$.  Note that the pp-type of $a$ in $M$ is generated by the set $\{\phi_n\}_n$ and by no finite subset of that.  Set $b=((0,1), (0,T), (0,T^2), \dots) \in B$.  For each $i$ we have $a\otimes (0,T^i) = aT^{-i}\otimes (0,1) = aT^{-i}\otimes (c,0)s$ and so $(a\otimes (0,T^i))_i \in (M\otimes B)^\omega$ is divisible by $s$, but clearly $a\otimes b \in M\otimes B^\omega$ is not divisible by $s$ and so the embedding of $M\otimes B^\omega$ into $(M\otimes B)^\omega$ is not pure.  Therefore, by \ref{MLpure}, we see that $M$ is not Mittag-Leffler with respect to $_RB_S$.

Therefore, being Mittag-Leffler with respect to $_RB_S$ is stronger than being Mittag-Leffler with respect to $_RB_{\mathbb Z}$, that is, than being $_RB$-ML.
\end{remark}

The conclusion of the next result has some resemblance to that of \ref{MLpure}.

\begin{prop}\label{tenspurproj} \marginpar{tenspurproj} Suppose that $_RB$ is a pure-projective left $R$-module with endomorphism ring $S$.  Then for any set $(M_\lambda)_\lambda$ of right $R$-modules, the canonical map $f:(\prod_\lambda \,M_\lambda)\otimes_RB_S \to \prod_\lambda \, (M_\lambda \otimes_RB_S)$ is a pure embedding of right $S$-modules.
\end{prop}
\begin{proof}  The map is an embedding because $_RB$ is Mittag-Leffler by \ref{MLchar}.  To check $S$-purity it is sufficient to show that, for every finitely presented left $S$-module $L$, the map $f\otimes_SL: (\prod_\lambda \,M_\lambda)\otimes_RB \otimes_SL \to \big(\prod_\lambda \, (M_\lambda \otimes_RB_S)\big) \otimes_SL$ is monic.  As seen in the proof of \ref{corML2}, since $_SL$ is finitely presented $\big(\prod_\lambda \, (M_\lambda \otimes_RB_S)\big) \otimes_SL \simeq \prod_\lambda \, (M_\lambda \otimes_RB\otimes_S L)$ so, to show that $f\otimes_SL$ is monic, it is sufficient to show that $_RB\otimes_SL$ is Mittag-Leffler.

Suppose first that $B$ is a direct sum of finitely presented modules.  Take a finite presentation $_SS^m \to \,_SS^n \to L \to 0$ of $L$, so we have an exact sequence $_RB\otimes_SS^m \to \,_RB\otimes_SS^n \to \,_RB\otimes_SL \to 0$.  Take any finite tuple $\overline{a}$ from $B\otimes_SL$.  Since $B$ is a direct sum of finitely presented modules, there is a finitely presented direct summand $F$ of $B$ such that $\overline{a}$ is contained in $F\otimes_SL$ which, since $\otimes$ commutes with direct sums, is a direct summand of $B\otimes_SL$; therefore the pp-type of $\overline{a}$ in $_RB\otimes_SL$ is equal to the pp-type of $\overline{a}$ in $_RF\otimes_SL$.  But, from the exact sequence $_RF\otimes_SS^m \to \,_RF\otimes_SS^n \to \,_RF\otimes_SL \to 0$, that is $_RF^m \to \,_RF^n \to \,_RF\otimes_SL \to 0$, we see that $F\otimes_SL$ is a finitely presented left $R$-module, and hence the pp-type of $\overline{a}$ in $_RF\otimes_SL$ is finitely generated.  Therefore, by \ref{MLchar}, $_RB\otimes_SL$ is a Mittag-Leffler left $R$-module.

For the general case, $_RB$ is a direct summand of a direct sum of finitely presented left $R$-modules so, referring to the proof above, $_RB\otimes_SL$ is a direct summand of a Mittag-Leffler left $R$-module and so is itself Mittag-Leffler, as required.
\end{proof}

\subsection{Tensoring by a bimodule}\label{sectensbimod} \marginpar{sectensbimod}

The functor $M_R \mapsto M\otimes_RB_S$ links definable subcategories of ${\rm Mod}\mbox{-}R$, ${\rm Mod}\mbox{-}S$ and $R\mbox{-}{\rm Mod}\mbox{-}S$, the last being the category of $(R,S)$-bimodules.\

\vspace{4pt}

For a definable subcategory ${\cal D}$ of ${\rm Mod}\mbox{-}R$, we will write ${\cal D} \otimes _RB_S$ for the definable subcategory of ${\rm Mod}\mbox{-}S$ generated by the collection of $M\otimes_RB_S$ with $M\in {\cal D}$.  We have seen already at the beginning of Section \ref{sectensextn} that this operation is not very well-behaved in general but, in the case that $_RB$ is finitely presented, we obtain a simple description of ${\cal D} \otimes_RB_S$.  First a lemma.

\begin{lemma}\label{tenspur} \marginpar{tenspur} If $M' \to M$ is a pure embedding of right $R$-modules and if $_RB_S$ is a bimodule, then $M'\otimes_RB_S \to M\otimes_RB_S$ is a pure embedding of right $S$-modules.

If $M$ is a right $R$-module and $B' \to B$ is a pure embedding of $(R,S)$-bimodules, then $M\otimes_RB'_S \to M\otimes_RB_S$ is a pure embedding of right $S$-modules.
\end{lemma}
\begin{proof}  For the first statement, we have the obvious algebraic proof, as follows.  An embedding of right modules is pure iff tensoring with any left module gives an embedding.  So it must be shown that $(M'\otimes_RB_S)\otimes_SL \to (M\otimes_RB_S)\otimes_SL$, that is $M'\otimes_R(B_S\otimes_SL) \to M\otimes_R(B_S \otimes_SL)$, is monic.  But that is so since $M' \to M$ is pure. 

We have an equally short proof using pp formulas.  Suppose that $M\otimes_RB_S \models \sigma(\overline{a} \otimes \overline{b})$ where $\sigma$ is a pp formula for $S$-modules, $\overline{a}$ is from $M'$ and $\overline{b}$ is from $B$.  By \ref{sigmaphi2} we have $M\models \phi(\overline{a})$ for some pp formula $\phi$ for $R$-modules and $_RB_S\models (\sigma: \phi)(\overline{b})$.  Since $M'$ is pure in $M$, we have $M' \models \phi(\overline{a})$ and so, by \ref{sigmaphi1}, we have $M'\otimes B \models \sigma(\overline{a} \otimes \overline{b})$, as required.

For the second statement, suppose that $M\otimes_RB_S \models \sigma(\overline{a} \otimes \overline{b})$ where $\sigma$ is a pp formula for $S$-modules, $\overline{a}$ is from $M$ and $\overline{b}$ is from $B'$.  By \ref{sigmaphi2} we have $M\models \phi(\overline{a})$ for some pp formula $\phi$ for $R$-modules and $_RB_S\models (\sigma: \phi)(\overline{b})$.  Since $B'$ is pure in $B$, we have $_RB'_S \models (\sigma: \phi)(\overline{b})$ and so, by \ref{sigmaphi1}, we have $M\otimes B' \models \sigma(\overline{a} \otimes \overline{b})$, as required.
\end{proof}

We will use the following observation.

\begin{lemma} \label{purclosdef}\marginpar{purclosdef} Suppose that ${\cal X}$ is a class of $R$-modules closed under direct products and directed colimits.  Then the closure of ${\cal X}$ under pure submodules is a definable subcategory of ${\rm Mod}\mbox{-}R$.
\end{lemma}
\begin{proof}. Let ${\cal X}^+$ denote the closure of ${\cal X}$ under pure submodules meaning that, $A \in {\cal X}^+$ iff there is a pure monomorphism $f:A \to B$ with $B\in {\cal X}$.  Certainly ${\cal X}^+$ is closed under pure submodules.  If $\{A_i\}_{i\in I} \subset {\cal X}^+$ then, for each $i\in I$ choose a pure monomorphism $f_i:A_i \to B_i\in {\cal X}$; then $\prod_i\, f_i : \prod_i \, A_i \to \prod_i \, B_i$ is a pure monomorphism, so $\prod_i \, A_i \in {\cal X}^+$.  A similar argument for directed colimits is a little awkward in that a directed system of the $B_i$ would have to be produced, so we use that it is enough, \ref{defchar}(iv), to check for closure under ultraproducts, having noted that ${\cal X}$ is closed under ultraproducts since, e.g. \cite[3.3.1]{PreNBK}, these are directed colimits of products.  So suppose that ${\cal U}$ is an ultrafilter on the index set $I$ and consider the morphism $\prod_i\, f_i /{\cal U}: \prod_i \, A_i /{\cal U} \to \prod_i \, B_i /{\cal U}$.  It is immediate from \L os' Theorem \ref{Los} that this is a monomorphism and is pure, as required.
\end{proof}

We will use the Jensen-Lenzing criterion for pure-injectivity.

\begin{theorem} \label{JeLe}\marginpar{JeLe} (\cite[7.1]{JeLe}) A module $N$ is pure-injective iff, for every (infinite) index set $I$, the summation map $\Sigma:N^{(I)} \to N$ factors through the canonical embedding $N^{(I)} \to N$.
\end{theorem}

We also need the following result of Lenzing.

\begin{theorem} \label{Lenfp}\marginpar{Lenfp} (\cite[S\"{a}tze 2]{Len}) A right $R$-module $M$ is finitely presented iff, for every collection $(L_i)_i$ of left $R$-modules, the canonical map $M\otimes_R\prod_iL_i \to \prod_i\, M\otimes_RL_i$ is an isomorphism.
\end{theorem}

\begin{lemma} \label{pitensfp}\marginpar{pitensfp} Suppose that $N_R$ is pure-injective and that $_RB_S$ is a bimodule with $_RB$ finitely presented.  Then $N\otimes_RB_S$ is a pure-injective $S$-module
\end{lemma}
\begin{proof}. By \ref{JeLe} if $I$ is any index set then there is a morphism $f$ making the diagram below commute, where $j$ is the natural embedding and $\Sigma$ is the summation map.
$$\xymatrix{N^{(I)} \ar[r]^{j_N} \ar[d]_{\Sigma_N} & N^I \ar[dl]^f \\ N
}$$
Tensor all this by $_RB_S$.  Then note that $N^{(I)} \otimes_RB$ is naturally isomorphic to $(N\otimes_RB)^{(I)}$ and that, by \ref{Lenfp}, since $_RB$ is finitely presented, the natural morphism $N^I\otimes_RB \to (N\otimes_RB)^I$ is an isomorphism.  This gives us a commutative diagram where, note, the maps all are $S$-linear.
$$\xymatrix{(N\otimes_RB)^{(I)} \ar[r]^{j_{N\otimes B}} \ar[d]_{\Sigma_{N\otimes B}} & (N\otimes_RB)^I \ar[dl]^{f\otimes B} \\ N
}$$
By \ref{JeLe} this implies that $N\otimes_RB_S$ is pure-injective.
\end{proof}

\begin{theorem}\label{extnpurproj} \marginpar{extnpurproj}  Suppose that the $(R,S)$-bimodule $B$ is finitely presented as a left $R$-module.  Let ${\cal D}$ be a definable subcategory of ${\rm Mod}\mbox{-}R$.  Then ${\cal D}\otimes_RB_S$ is the closure of $\{ M\otimes_RB_S: M\in {\cal D}\}$ under $S$-pure submodules and, if $N$ is an elementary cogenerator of ${\cal D}$, then $N\otimes_RB_S$ is an elementary cogenerator of ${\cal D} \otimes_RB_S$.
\end{theorem}
\begin{proof}  The functor $-\otimes_RB_S$ commutes with directed colimits.  Also, by \ref{Lenfp}, for all right $R$-modules $M_i$, the canonical morphism $(\prod_i \, M_i)\otimes _RB \to \prod_i \, (M_i\otimes_RB)$ is a pure embedding of right $S$-modules.  So by \ref{purclosdef} the only operation required to extend $\{ M\otimes_RB_S: M\in {\cal D}\}$ to a definable subcategory is closure under $S$-pure submodules.

For the second statement, every $M\in {\cal D}$ is a pure submodule of some power $N^I$ of $N$, so we have, by \ref{tenspur}, $M\otimes _RB_S$ pure in $N^I \otimes_RB_S$ which, by \ref{tenspurproj}, is $S$-pure in $(N\otimes_RB_S)^I$.  Therefore any $S$-pure submodule of $M\otimes_RB_S$ will be $S$-pure in a product of copies of $N\otimes_RB_S$, as claimed.
\end{proof}

\vspace{4pt}

If ${\cal D}$ is a definable subcategory of ${\rm Mod}\mbox{-}R$ and ${\cal E}$ is a definable subcategory of ${\rm Mod}\mbox{-}S$, then denote by $({\cal E}: {\cal D})$ the category of $(R,S)$-bimodules which, by tensor, send ${\cal D}$ into ${\cal E}$:
$$({\cal E}: {\cal D}) = \{ _RB_S: \forall \, D\in {\cal D}, \,\, D\otimes_RB_S \in {\cal E}\}.$$
It is proved in \cite[2.3]{PreTens} that $({\cal E}: {\rm Mod}\mbox{-}R)$ is a definable subcategory of $R\mbox{-}{\rm Mod}\mbox{-}S$.  We generalise that here as follows.

\begin{theorem}\label{senddef} \marginpar{senddef} Suppose that ${\cal D}$ is a definable subcategory of ${\rm Mod}\mbox{-}R$ which is $\varinjlim$-generated by modules which are Mittag-Leffler with respect to all $(R,S)$-bimodules (for example, suppose that ${\cal D}$ is generated by the finitely presented modules in it).  Let ${\cal E}$ be a definable subcategory of ${\rm Mod}\mbox{-}S$.  Then $({\cal E}: {\cal D})$ is a definable subcategory of $R\mbox{-}{\rm Mod}\mbox{-}S$.
\end{theorem}
\begin{proof}  In contrast to the proof in \cite{PreTens} we give an algebraic proof here.  If $B =\varinjlim \, B_i$ is a directed colimit of bimodules $B_i \in ({\cal E}: {\cal D})$ and if $M\in {\cal D}$, so each $M\otimes B_i \in {\cal E}$ then, since $M\otimes -$ commutes with directed colimits, we have $M\otimes B \in {\cal E}$ and hence $B \in ({\cal E}: {\cal D})$.  If $B'$ is a pure subbimodule of $B \in ({\cal E}: {\cal D})$, then, by \ref{tenspur}, if $M\in {\cal D}$, we have $M\otimes B'$ pure in $M\otimes B$, hence $M\otimes B' \in {\cal E}$ and so $B' \in ({\cal E}: {\cal D})$.

Suppose that $M\in {\cal D}$ is Mittag-Leffler for $R\mbox{-}{\rm Mod}\mbox{-}S$ and that $B_i \in ({\cal E}: {\cal D})$, $i\in I$.  The embedding $M\otimes \prod_i\, B_i \to \prod M\otimes B_i$ is $S$-pure by \ref{MLpure}, so $M\otimes \prod_i\, B_i \in {\cal E}$, and therefore $\prod_i B_i \in ({\cal E}: {\cal D})$.  Now take any $M\in {\cal D}$.  By assumption $M = \varinjlim_j M_j$ with the $M_j$ $R\mbox{-}{\rm Mod}\mbox{-}S$-Mittag-Leffler modules in ${\cal D}$.  We have $M\otimes \prod_i\, B_i = (\varinjlim \, M_j) \otimes \prod_i\, B_i = \varinjlim \, (M_j \otimes \prod_i\, B_i)$ and this is in ${\cal E}$ since each $M_j \otimes \prod_i\, B_i$ is in ${\cal E}$.  Therefore $({\cal E}: {\cal D})$ is also closed under direct products, so is definable as claimed.
\end{proof}

That some restriction on ${\cal D}$ is needed is shown by the following example.

\begin{example}\label{senddefex} \marginpar{senddefex} \cite[after 2.3]{PreTens} Take $R=S={\mathbb Z}$, ${\cal D} = \langle {\mathbb Q}\rangle$, ${\cal E} =0$.  Then every torsion module is in $({\cal E}: {\cal D})$, yet ${\mathbb Q}$, which is in the definable subcategory generated by the torsion modules, is not.
\end{example}

\vspace{4pt}

For ${\cal D}$ a definable subcategory of ${\rm Mod}\mbox{-}R$ and ${\cal B}$ a definable subcategory of $R\mbox{-}{\rm Mod}\mbox{-}S$, we define ${\cal D} \otimes {\cal B}$ to be the definable subcategory of ${\rm Mod}\mbox{-}S$ generated by the class of modules of the form $M\otimes_RB_S$ with $M\in {\cal D}$ and $B \in {\cal B}$.  This `tensor' operation on definable subcategories was looked at in \cite[\S 3]{PreTens}, though done there in terms of theories of modules and closed subsets of the Ziegler spectrum, rather than in terms of definable subcategories.  Here we reconsider and extend those results.

In view of \ref{senddefex} it is necessary to impose some conditions on $M$ and $B$ in order for $M\otimes B$ to be a definable generator for $\langle M \rangle \otimes \langle B \rangle$.  It was shown in \cite[3.1]{PreTens} that, if $M$ and $B$ are weakly saturated, then $\langle M\otimes B \rangle = \langle M\rangle \otimes \langle B \rangle$, where $M$ is {\bf weakly saturated} if, for every $M'\in \langle M \rangle $ and every tuple $\overline{b}$ from $M'$, there is $\overline{a}$ from $M$ with ${\rm pp}^M(\overline{a}) ={\rm pp}^{M'}(\overline{b})$.\footnote{Provided every pp-quotient $\phi(M)/\psi(M)$ is infinite if nonzero, then this is equivalent to the usual model-theoretic definition of weak saturation.}

We will say that a module $M\in {\cal D}$ is a {\bf pure cogenerator} for the definable subcategory ${\cal D}$ (necessarily ${\cal D} = \langle M \rangle$) if every $D\in {\cal D}$ purely embeds in a power $M^I$ of $M$.  Note that an elementary cogenerator for ${\cal D}$ is exactly a pure-injective pure cogenerator for ${\cal D}$.  These various notions of largeness in ${\cal D}$ are quite closely related.

\begin{lemma} Let ${\cal D} = \langle M \rangle$ be the definable subcategory generated by a module $M$.

\noindent (1)  If $M$ is a pure cogenerator for ${\cal D}$ then its pure-injective hull is an elementary cogenerator for ${\cal D}$.

\noindent (2) If $M$ is a pure cogenerator then some power of $M$ is weakly saturated.

\noindent (3) If $M$ is weakly saturated and pure-injective, then $M$ is an elementary cogenerator.
\end{lemma}
\begin{proof} (1) is by the definitions.  For (2), since each member of $\langle M \rangle$ is purely embedded in some power of $M$, each pp-type realised in some member of $\langle M \rangle$ is realised in some power of $M$.  Since there is just a set of pp-types, there is a power of $M$ large enough to realise every such pp-type, so that power is weakly saturated.  For (3), let $M' \in \langle M \rangle$.  For each element $a\in M'$, there is, by hypothesis, some $b\in M$ such that ${\rm pp}^M(b) = {\rm pp}^{M'}(a)$.  By \cite[4.3.9]{PreNBK}, there is, since $M$ is pure injective, a morphism $f_a:M' \to M$ with $fa=b$.  Consider $\prod_{a\in M'} \, f_a:M' \to M^{M'}$.  By construction this is a pure embedding and so $M$ is an elementary cogenerator.
\end{proof}

For a definable subcategory ${\cal D}$, say that $M\in {\cal D}$ is {\bf finitely generic} for ${\cal D}$ if, for every pp formula $\phi$, there is $\overline{a}$ from $M$ such that ${\rm pp}^M(\overline{a})$ is ${\cal D}$-generated by $\phi$, that is, ${\rm pp}^M(\overline{a}) = \{ \phi': \phi \leq_{\cal D} \phi'\}$.

\begin{lemma}\label{existfingen} \marginpar{existfingen} Every definable category ${\cal D}$ contains a ${\cal D}$-atomic finitely generic module for ${\cal D}$.

If $M$ is a pure cogenerator for ${\cal D}$, then some power of $M$ is finitely generic for ${\cal D}$.
\end{lemma}
\begin{proof} If $\phi$ is a pp formula then there is a ${\cal D}$-atomic module $M_\phi$ in ${\cal D}$ which contains a tuple with pp-type ${\cal D}$-generated by $\phi$ - \cite{Makk}, see \cite[4.10]{PreStrAt}.  Take the direct sum of these modules as $\phi$ ranges over pp formulas.  Then this is ${\cal D}$-atomic (\cite[4.4]{PreStrAt}) and, by construction, is finitely generic.

For the second statement, if $M$ is a pure cogenerator for ${\cal D}$ and if $\phi$ is a pp formula, then $M_\phi$ as above purely embeds in some power of $M$.  Since there is just a set of pp formulas, a suitable large power of $M$ embeds all such $M_\phi$, and hence is finitely generic for ${\cal D}$.
\end{proof}

Note that a ${\cal D}$-atomic finitely generic module for ${\cal D}$ will not normally be an elementary cogenerator, or even a pure cogenerator.  For instance, in the case ${\cal D} = {\rm Mod}\mbox{-}R$, the direct sum of one copy of each finitely presented module will be finitely generic, whereas a pure cogenerator must have, as a direct summand, a copy of the injective hull of each simple module.

\begin{prop}\label{tenspp} \marginpar{tenspp} (\cite[2.4]{PreTens}) If ${\cal D}$ a definable subcategory of ${\rm Mod}\mbox{-}R$ and ${\cal B}$ a definable subcategory of $R\mbox{-}{\rm Mod}\mbox{-}S$, then the following are equivalent:

\noindent (i) the pp-pair $\sigma /\tau$ is closed on ${\cal D} \otimes {\cal B}$;

\noindent (ii) for all pp formulas $\phi$ for $R$-modules, there is a pp formula $\phi'$ for $R$-modules such that $\phi \leq_{\cal D} \phi'$ and such that $(\sigma: \phi) \leq_{\cal B} (\tau: \phi')$;
\end{prop}
\begin{proof} (ii)$\Rightarrow$(i) Suppose that $B\in {\cal B}$ and $B\models \sigma(\overline{a} \otimes \overline{b})$ where $\overline{a}$ is from $M\in {\cal D}$ and $\overline{b}$ is from $B$.  Then $M\models \phi(\overline{a})$ and $B\models (\sigma: \phi)(\overline{b})$ for some pp formula $\phi$ for $R$-modules.  By assumption there is $\phi'$ such that $\phi \leq_{\cal D} \phi'$, so $M\models \phi'(\overline{a})$, and also $B\models (\tau: \phi')(\overline{b})$.  Therefore, by \ref{sigmaphi1}, we have $M\otimes B \models \tau(\overline{a} \otimes \overline{b})$, as required.

\noindent (i)$\Rightarrow$(ii) Suppose that $\sigma/\tau$ is closed on ${\cal D} \otimes {\cal B}$.  Choose, by \ref{existfingen}, $M$ finitely generic in ${\cal D}$ and $B$ finitely generic in ${\cal B}$.  Let $\phi$ be a pp formula for $R$-modules and choose $\overline{a}$ whose pp-type in $M$ is ${\cal D}$-generated by $\phi$.  Also choose $\overline{b}$ from $B$ whose pp-type in $B$ is ${\cal B}$-generated by $(\sigma: \phi)$.  Certainly $M\otimes B \models \sigma(\overline{a} \otimes \overline{b})$; by assumption $M\otimes B \models \tau(\overline{a} \otimes \overline{b})$.  Therefore there is $\phi' \in {\rm pp}^M(\overline{a})$ such that $(\tau: \phi') \in {\rm pp}^B(\overline{b})$.  By choice of $\overline{a}$, and $\overline{b}$ we have $\phi \leq_{\cal D} \phi'$ and $(\sigma: \phi) \leq_{\cal B} (\tau: \phi')$, as required.
\end{proof}

\begin{theorem}\label{tensgen} \marginpar{tensgen} (cf.~\cite[3.1]{PreTens}) If $M$ is finitely generic in the definable subcategory ${\cal D}$ of ${\rm Mod}\mbox{-}R$ and $B$ is finitely generic in the definable subcategory ${\cal B}$ of $R\mbox{-}{\rm Mod}\mbox{-}S$, then $\langle M\otimes_RB_S \rangle = {\cal D} \otimes {\cal B}$.
\end{theorem}
\begin{proof} Certainly $M\otimes_RB_S \in {\cal D} \otimes {\cal B}$.  For the converse, suppose that $\sigma/\tau$ is open on ${\cal D} \otimes {\cal B}$.  By \ref{tenspp} there is a pp formula $\phi$ for $R$-modules such that 

\noindent ($\ast$) for every $\phi'$ with $\phi \leq_{\cal D} \phi'$ we do not have $(\sigma: \phi) \leq_{\cal B} (\tau: \phi')$.  

\noindent By assumption, there is $\overline{a}$ from $M$ with the pp-type of $\overline{a}$ in $M$ being ${\cal D}$-generated by $\phi$.  Also, by assumption there is $\overline{b}$ from $B$ with the ${\rm pp}^B(\overline{b})$ being ${\cal B}$-generated by $(\sigma: \phi)$.  Certainly then $M\otimes B \models \sigma(\overline{a} \otimes \overline{b})$.  If we had $M\otimes B \models \tau(\overline{a} \otimes \overline{b})$, then there would be $\phi' \in {\rm pp}^M(\overline{a})$ such that $(\tau:\phi') \in {\rm pp}^B(\overline{b})$.  By choice of $\overline{a}$, $\phi \leq_{\cal D} \phi'$ and by choice of $\overline{b}$, $(\sigma: \phi) \leq_{\cal B} (\tau: \phi')$, contradicting ($\ast$).  Therefore $\sigma/\tau$ is open on $M\otimes B$, as required (as commented after \ref{defchar}).
\end{proof}

\section{Extension along elementary embeddings} \label{secextindelem} \marginpar{secextindelem}

\subsection{Elementary extensions of rings and algebras} \label{secelemextrng} \marginpar{secelemextrng}

In this section we will suppose that $f:R \to S$ is an elementary extension of rings.  For instance, suppose that $R$ is a finite-dimensional algebra over an algebraically closed field $K$ and let $L$ be an algebraically closed extension field of $K$.  Then the embedding $R \to S = R\otimes_KL$ is\footnote{Sketch proof: choose a $K$-basis $(a_1,\dots,a_n)$ for $R$ over $K$ and let the $\gamma_{ijk} \in K$ be such that $a_ia_j =\sum_k\, \gamma_{ijk}a_k$.  Then $R$ can be ``coded up" within $K$, being identified with the vector space $K^n$ endowed with the multiplication $(\alpha_1, \dots, \alpha_n) \cdot (\beta_1, \dots, \beta_n) = (\sum_{ij}\, \gamma_{ij1}\alpha_i\beta_j, \dots, \sum_{ij}\, \gamma_{ijn}\alpha_i\beta_j)$.  Since $L$ is an elementary extension of $K$ (e.g.~\cite[A.5.1]{Hod}), the resulting ``coded up" structure within $L$, namely $S$, is an elementary extension of $R$; cf.~\cite[1.1]{Rose}.} an example of an {\bf elementary extension} of rings, meaning that if $\sigma$ is a sentence (a formula without free variables) in the language for rings and possibly naming parameters from $R$, then $\sigma$ is true in $R$ iff it is true in $S$.  In this situation we also say that $R$ is an {\bf elementary subring} of $S$ and write $R \prec S$.  For instance, if $R\to S$ is an elementary extension, then $r\in R$ belongs to the Jacobson radical of $R$ iff it belongs to the Jacobson radical of $S$, because we can take $\sigma$ to be the sentence $\forall y \exists z \, ((1+ry)z =1 = z(1+ry))$.  This is a stronger condition than $R$ and $S$ being {\bf elementarily equivalent} which requires only that they satisfy the same sentences (without extra parameters from $R$) in the language of rings.

In fact, for many of the results we require only that the embedding $R \to S$ be an elementary extension of $(R,R)$-bimodules (the sentences used in the proofs will involve multiplications of elements of $R$ with elements of $S$ but not arbitary multiplications of elements of $S$).

\begin{theorem}\label{elemextup} \marginpar{elemextup} (e.g.~\cite[9.5.2]{Hod} and then \ref{isoup} below for the second statement) If $S$ is any first-order structure (a module, a ring, ...) and ${\cal U}$ is an ultrafilter on an index set $I$, then the diagonal embedding $s \mapsto (s)_i/{\cal U}$ from $S$ to its ultrapower $S^I/{\cal U}$, is an elementary extension (of modules, rings, ...).

If $S \to S'$ is an elementary embedding of structures, then $S'$ is an elementary substructure of some ultrapower of $S$.
\end{theorem}

If $R$ is a finite-dimensional algebra over a field $K$ and $S$ is an elementary extension of $R$, then $S$ need not be a finite-dimensional algebra over a field.  For an example, take $R=Q\times Q$ where $Q$ is the field of algebraic complex numbers.  Then $Q$ is an elementary substructure of any algebraically closed field of characteristic $0$ (e.g.~\cite[A.5.1]{Hod}), in particular $Q \to {\mathbb C}$ is an elementary embedding.  By the Feferman-Vaught Theorem (see \cite[9.6.5(b)]{Hod}) $Q\times Q \to Q\times {\mathbb C}$ is an elementary embedding but clearly the latter is not a finite-dimensional algebra over any field.  However, for elementary extensions that are ultrapowers, the situation is nicer and we will make use of the above fact \ref{elemextup} that, if $R$ is a finite-dimensional algebra and $R \to S$ is an elementary extension, then $S$ is an elementary subring of some ultrapower of $R$ and those ultrapowers have a simple description. 

\begin{lemma}\label{findimup} \marginpar{findimup} Suppose that $R$ is a finite-dimensional algebra over a field $K$, with $K$-basis $\beta_1, \dots, \beta_t$.  Let $R^\ast = R^I/{\cal U}$ be any ultrapower of $R$.  Then $R^\ast$ is an algebra over the field $K^I/{\cal U}$ with basis $\beta_1,\dots, \beta_t$ (and the algebraic relations between $\beta_1,\dots, \beta_t$ as elements of $R^\ast$ are generated by the relations between them as elements of $R$). 
\end{lemma}
\begin{proof} Recall (cf.~e.g.~\cite[3.3.1]{PreNBK}) that $R^I/{\cal U}$ can be constructed as the factor ring of $R^I$ by the ideal of elements $(r_i)_i$ with $\{ i\in I: r_i=0\} \in {\cal U}$.  Every element $r^\ast$ of $R^I/{\cal U}$ has the form $(r_i)_i/{\cal U}$ where $r_i = \sum_{j=1}^t \, \alpha_{ij} \beta_j$, so $r^\ast$ may be written as $\sum_{j=1}^t\, (\alpha_{ij})_i/{\cal U} \cdot \beta_j$, where we are identifying the elements $\beta_j \in R$ with their images $(\beta_j)_i/{\cal U}$ under the diagonal embedding of $R$ into $R^\ast$.  Thus (the images of) $\beta_1,\dots, \beta_t$ in $R^\ast$ form a generating set for $R^\ast$ over $K^\ast = K^I/{\cal U}$.  To see that they are linearly independent over $K^\ast$, suppose that $\sum_i^t \, \alpha^\ast_j\beta_j =0$ where $\alpha_j^\ast = (\alpha_{ij})_i/{\cal U}$.  Then the set of $i\in I$ such that $\sum_1^t\, \alpha_{ij}\beta_j =0$ is in ${\cal U}$; since the $\beta_j$ are linearly independent over $K$, the $\alpha_{ij}$ appearing in those zero relations all are $0$; in particular, for each $j$, $\{ i\in I: \alpha_{ij}=0\}$ is in ${\cal U}$ and hence each $\alpha_j^\ast=0$, as required.  The proof of the parenthetical statement is similar.
\end{proof}

\begin{remark}\label{fdalgbimod} \marginpar{fdalgbimod} Note that if $R$ is a finite-dimensional algebra over a field $K$, then the bimodule $_RR_R$, that is, $R$ as a right module over the $K$-algebra $R^{\rm op}\otimes_K R$, being of finite length as a $K$-module, has both acc and dcc on pp-definable subgroups.
\end{remark}

\begin{theorem}\label{isoup} \marginpar{isoup} (Keisler-Shelah Theorem) \cite{SheIso} If $M$, $N$ are structures (for the same first order, finitary language), then $M$ and $N$ are elementarily equivalent if and only if they have isomorphic ultrapowers $M^I/{\cal U} \simeq N^J/{\cal V}$ for some sets $I$, $J$ and ultrafilters ${\cal U}$, ${\cal V}$ on them\footnote{Indeed one may even require $I=J$ and ${\cal U} = {\cal V}$.}.
\end{theorem}

Another useful result is the Downwards L\"{o}wenheim-Skolem Theorem which, for rings, and in the countable version, is the following.

\begin{theorem}\label{downLSrng} \marginpar{downLSrng} Suppose that $X$ is a countable subset of the ring $S$.  Then there is a countable elementary subring of $S$ which contains $X$.
\end{theorem}

The more general theorem is as follows.

\begin{theorem}\label{downLS} \marginpar{downLS} Suppose that $S$ is a first-order ${\cal L}$-structure where the cardinality of the language ${\cal L}$ is $\kappa$; suppose that $X$ is a subset of $S$.  Then there is an elementary substructure of $S$ which contains $X$ and is of cardinality $\leq |X| + \kappa$.
\end{theorem}

By the cardinality of a language we mean $\aleph_0$ plus the number of extra (constant, function, relation) symbols in the language.  For instance the cardinality of the language of rings is $\aleph_0$, the cardinality of the language of $R$-modules is $\aleph_0 + |R|$, and the cardinality of the language for treating elementary extensions of a ring $R$ is $\aleph_0 + |R|$ (because the sentences in the definition of elementary extension may use parameters from $R$). 

\vspace{4pt}

We recall the condition (\cite[8.10]{PreBk}, see \cite[1.1.13]{PreNBK}) for one pp formula to imply another.

\begin{theorem}\label{presta} \marginpar{presta} Suppose that $\psi(\overline{x})$ and $\phi(\overline{x})$ are pp formulas in the language of $R$-modules; say $ \psi (\overline{x})$ is $\exists \overline{z}\, \big((\overline{x} \, \overline{z})H_\psi =0\big) $ and $\phi$ is $\exists  \overline{y}\, \big((\overline{x} \, \overline{y})H_\phi =0 \big)$, where $H_\phi$ and $H_\psi$ are matrices with entries in $R$.  Then $\psi \leq \phi$ iff there are matrices $ G= \left(\begin{array}{c} G' \\ G'' \end{array} \right) $ and $ K $, with entries in $R$, such that
$$ \left( \begin{array}{cc} I & G' \\ 0 & G'' \end{array} \right)   H_\phi =H_\psi K $$ (here $ I $ is the $ n\times n $ identity matrix where $ n $ is the length of $ \overline{x} $, and $ 0$ denotes a zero matrix with $ n $ columns).
\end{theorem}

We noted, at the beginning of Section \ref{secdirextn}, that if $R\to S$ is a morphism of rings, then, for each $n$ there is an induced lattice homomorphism ${\rm pp}^n_R \to {\rm pp}^n_S$, where ${\rm pp}^n_R$ is the lattice of equivalence classes of pp formulas for right $R$-modules with $n$ free variables.  In the case of an elementary embedding, this is an embedding of lattices.

\begin{cor}\label{elemincl} \marginpar{elemincl} (\cite[Prop.~3]{PreRepEE}) If $f:R \to S$ is an elementary embedding of rings (or just of $(R,R)$-bimodules), then $\phi \to f_\ast\phi$ induces, for each $n$, an embedding of the lattice ${\rm pp}^n_R$ into the lattice ${\rm pp}^n_S$.  In particular $\psi \leq \phi$ for $R$-modules iff $f_\ast\psi \leq f_\ast\phi$ for $S$-modules.
\end{cor}
\begin{proof}  The direction $\Rightarrow$ in the second sentence needs no assumptions and is \ref{fstarleq}.  For the other direction, suppose that $\phi$, $\psi$ are pp formulas for $R$-modules and that $f_\ast\psi \leq f_\ast\phi$.  Then there are matrices $G$ and $K$ with entries in $S$ making the matrix equation in the statement of \ref{presta} true.  We can write that matrix equation as a conjunction, $\theta(\overline{r}, \overline{s})$, of equations in the language of rings and with parameters $\overline{r}$ - the elements of $R$ which appear in $H_\phi$ and $H_\psi$ - and parameters $\overline{s}$ - the elements of $S$ which appear in $G$ or $K$.  Existentially quantifying out the parameters $\overline{s}$, we obtain $\exists \overline{y} \, \theta(\overline{r}, \overline{y})$ - a sentence in the language of rings, indeed a sentence in the language of $(R,R)$-bimodules, with parameters from $R$.  That sentence is true in $S$ so, since $R$ is an elementary subring (or subbimodule) of $S$, it is true in $R$.  That is, there are matrices $G'$, $K'$ with entries in $R$ which solves the matrix equation in the statement of \ref{presta}, and so $\psi \leq \phi$ for $R$-modules, as required.
\end{proof}

\begin{lemma} \label{elem3} \marginpar{elem3} (\cite[Prop.~18]{PreRepEE}) Suppose that $R \prec S$ (as rings, or just as $(R,R)$-bimodules).  If $M \in {\rm Mod}\mbox{-}R$, then the canonical $R$-linear map $M \to M\otimes_R S$, given by $m\mapsto m\otimes 1_S$, is a pure embedding of $R$-modules.  In particular, $M$ is in the definable subcategory of ${\rm Mod}\mbox{-}R$ generated by $M\otimes_RS_R$.
\end{lemma}
\begin{proof}  If the map $a \mapsto a\otimes 1_S$ from $M$ to $M\otimes_RS$ were not an embedding, then there would be some $a_1, \dots, a_n$ in $M$ and $r_1, \dots r_n \in R$ such that $\sum_{i=1}^n \, a_ir_i$ is nonzero in $M$ but $\sum_{i=1}^n \, a_i\otimes r_i =0$ in $M \otimes_R S$.  By Herzog's criterion, \ref{herzcrit}, there would then be a pp formula $\phi(\overline{x})$ for $R$-modules such that $M_R \models \phi(\overline{a})$ and $_{R}S \models D\phi(\overline{r})$.  Since the embedding $R \to S$ of left $R$-modules is pure (since it is elementary), it follows that $_RR \models D\phi(\overline{r})$ and hence, again by \ref{herzcrit}, $\sum_{i=1}^n \, a_i\otimes r_i = \sum_{i=1}^n \, a_ir_i =0$ already in $R$.  Therefore $M \to M \otimes_R S$ given by $a\mapsto a\otimes 1_S$ is monic.

To see that, more generally, this map is a pure embedding of right $R$-modules, suppose that $M\otimes_RS \models \phi(\sum_{i=1}^n \, a_i\otimes r_i)$ ($a_i$, $r_i$ as above) where $\phi$ is a pp formula in the language of $R$-modules.  Then, by \ref{sigmaphi2}, we have $M\models \psi(\overline{a})$ and $_{R}S_S \models (\phi: \psi)(\overline{r})$ for some pp $\psi$ in the language of $R$-modules.  As noted before \ref{elem4}, since $\phi$, $\psi$ both are in the language of $R$-modules, $(\phi:\psi)$ is in the language of $(R,R)$-bimodules, so we have $_{R}S_R\models (\phi: \psi)(\overline{r})$ and hence, since $R\prec S$, we have $_{R}R_R\models (\phi: \psi)(\overline{r})$.  By Herzog's criterion again, we deduce that already $M \simeq M \otimes_{R}R \models \phi(\sum_{i=1}^n \, a_i\otimes r_i)$, as required.
\end{proof}

We note the following.

\begin{prop} Suppose that $R \prec S$ as rings.

\noindent (1) If $R$ is right coherent, then $_RS$ is flat but not necessarily Mittag-Leffler.

\noindent (2) If $R$ is right coherent and left perfect, then $_RS$ is projective (and hence Mittag-Leffler).
\end{prop} 
\begin{proof}  Since $R\prec S$ implies that $_RS$ is elementarily equivalent to $_RR$, the positive statements follow by \cite[Thms.~4, 5]{SabEk} (see \cite[Chpt.~14]{PreBk}).  An example for the negative statement in (1) is to take $R$ to be the localisation ${\mathbb Z}_{(p)}$ of ${\mathbb Z}$ at a prime $p$ and consider some nonprincipal ultrafilter ${\cal U}$ on the index set $\omega$ (start with the filter generated by all subsets of $\omega$ with finite complement and apply Zorn's Lemma to extend to an ultrafilter).  Consider the element $a = (p, p^2, \dots, p^n, \dots)/{\cal U}$ of $S= R^\omega/{\cal U}$.  Then $p^n|a$ for every $n$ (by {\L}os' Theorem, \ref{Los}, but this follows directly from the ultrapower construction) and so ${\rm pp}^S(a)$ is not finitely generated and $_RS$ is not Mittag-Leffler, despite $S$ being an elementary extension of $R$.
\end{proof}

We will concentrate on tensor extension but note the following regarding direct extension.

If $R$ is an elementary subring of $S$ then, in general, there are many more definable subcategories of ${\rm Mod}\mbox{-}S$ than of ${\rm Mod}\mbox{-}R$:  for example over tame hereditary algebras not of finite representation type, there are series of definable subcategories parametrised by the projective line over the underlying field (which field will be very large if $S$ is a large ultrapower of $R$).  On the other hand, if these rings are of finite representation type, then every definable subcategory of ${\rm Mod}\mbox{-}S$ is the direct extension of a definable subcategory of ${\rm Mod}\mbox{-}R$.  To see that, first recall that elementary equivalence of rings preserves finite representation type.

\begin{prop}\label{eefrt} \marginpar{eefrt} (\cite[2.5]{HJL})  If $R$ and $S$ are elementarily equivalent rings and if $R$ is of finite representation type, then so is $S$.
\end{prop}

Indeed in \cite{HJL} it is shown that much is preserved by elementary equivalence between rings of finite representation type, including the Auslander-Reiten quiver and so, in particular, the number of indecomposable modules.  Since those indecomposables are the points of the Ziegler spectrum and since the topology is discrete in the case of finite representation type (and since the (closed) subsets correspond bijectively to the definable subcategories), we obtain the following corollary.

\begin{cor}\label{eefretdefsub} \marginpar{eefrtdefsub} If $R$ is of finite representation type and $R \prec S$ as rings, then direct extension of definable subcategories gives a bijection between the definable subcategories of ${\rm Mod}\mbox{-}R$ and the definable subcategories of ${\rm Mod}\mbox{-}S$.
\end{cor}

In fact, the methods used in this paper give another proof of \ref{eefrt}, see \cite[Cor.~8]{PreRepEE}.

\subsection{Tensor extension and restriction of definable subcategories} \label{secdirelem} \marginpar{secdirelem}

{\bf Throughout this section we assume that $R$, $S$ are rings}.  Given $R\prec S$, we take a definable subcategory ${\cal D}$ of ${\rm Mod}\mbox{-}R$ and compare the tensor, ${\cal D}\otimes_RS$, and direct, ${\cal D}^S$, extensions of ${\cal D}$.  We also compare the restriction of ${\cal D} \otimes_RS$ to ${\rm Mod}\mbox{-}R$ with the original subcategory ${\cal D}$.  

Again, many of our proofs work just assuming that $R$, $S$ are rings and that $R \to S$ is an elementary embedding of $(R,R)$-bimodules so in this section we often assume just that $R$ is an elementary sub-$(R,R)$-bimodule of $S$.

The point is that in the bimodule language we cannot multiply together arbitrary elements of the ring, though we can multiply arbitrary elements, on the right or on the left, by any specific element of $R$.  For instance, if $R$ is commutative and if $R \prec S$ is an elementary extension of rings, then also $S$ will be commutative since the formula (in the language of rings) $\forall x, y \, (xy=yx)$ is true in $R$, so it will be true in $S$.  On the other hand if $_RR_R \prec _RS_R$ is an elementary extension of $(R,R)$-bimodules, then, for each $r\in R$ we have the sentence (with parameter $r$, in the bimodule language) $\forall y (ry =yr)$ true in $R$ and so true in $S$, but this says just that $R$ is in the centre of $S$.

\begin{prop}\label{modextn} \marginpar{modextn} If $f:R\to S$ is an elementary extension of rings, or just of $(R,R)$-bimodules, then ${\rm DefTr}^S_R = {\rm Mod}\mbox{-}R$.
\end{prop}
\begin{proof} By assumption, $R$ and $S$ are elementarily equivalent as structures with the elements of $R$ named as parameters in the language, so by \ref{isoup} they have isomorphic ultrapowers:  $\alpha: S^J/{\cal V} \simeq R^I/{\cal U}$ for some index sets $I$, $J$ and ultrafilters ${\cal U}$ on $I$ and ${\cal V}$ on $J$ such that $\alpha$ preserves the diagonal images of $R$, meaning that, for each $r\in R$, $\alpha \big((fr)_j/{\cal V}\big) = (r)_i/{\cal U}$, where $(fr)_j/{\cal V}$ and $(r)_i/{\cal U}$ are the constant sequences $fr$, respectively $r$.  

Let $M$ be any $R$-module.  Note that $M^I/{\cal U}$ has a natural action as a right $R^I/{\cal U}$-module, namely $(a_i)_i/{\cal U} \cdot (r_i)_i/{\cal U} = (a_ir_i)_i/{\cal U}$, so it also is a right $S^J/{\cal V}$-module by $(a_i)_i/{\cal U} \cdot (s_j)_j/{\cal V} = (a_i)_i/{\cal U} \cdot \alpha \big((s_j)_j/{\cal V}\big)$.  This action of $S^J/{\cal V}$ extends the action of $R$ on $M^I/{\cal U}$ since $(a_i)_i/{\cal U} \cdot (fr)_j/{\cal V} = (a_i)_i/{\cal U} \cdot \alpha \big((fr)_j/{\cal V}\big) = (a_i)_i/{\cal U} \cdot (r)_i/{\cal U} = (a_ir)_i/{\cal U}$.  Hence the action of $S$ on $M^I/{\cal U}$ {\it via} the diagonal embedding $S \to S^J/{\cal V}$ extends the action of $R$ on $M^I/{\cal U}$.  So $(M^I/{\cal U})_R \in {\rm DefTr}^S_R$.  The diagonal embedding of $M$ into $M^I/{\cal U}$ is an elementary embedding of $R$-modules (\ref{elemextup}), in particular a pure embedding, so $M \in {\rm DefTr}^S_R$ and ${\rm DefTr}^S_R = {\rm Mod}\mbox{-}R$, as required.
\end{proof}

\begin{lemma}\label{elemindnrestr} \marginpar{elemindnrestr}  Suppose that $f:R\to S$ is an elementary embedding of $(R,R)$-bimodules and let ${\cal D}$ be a definable subcategory of ${\rm Mod}\mbox{-}R$.  Then $({\cal D}^S)|_R = {\cal D}$.
\end{lemma}
\begin{proof} This is immediate from \ref{deftrextn} and \ref{modextn}.
\end{proof}

There are various conditions under which $M_R$ and $M\otimes_RS_R$ generate the same definable subcategory and hence are definably equivalent.  For instance, we have the following from \cite{PreRepEE}.

\begin{cor}\label{MotimesSR} \marginpar{MotimesSR} (\cite[Prop.~21]{PreRepEE}) Suppose that $R\prec S$ as $(R,R)$-bimodules and that $M$ is a right $R$-module.  Assume either that $M_R$ is Mittag-Leffler or that the bimodule $_RR_R$ has the ascending chain condition on pp-definable subgroups.  Then $M$ and $M\otimes_RS_R$ are definably equivalent
\end{cor}
\begin{proof}  We recall the proofs.  In view of \ref{elem3} we have to show that $M\otimes_RS_R$ is in the definable subcategory generated by $M_R$.  So suppose that $\sigma/\tau$ is a pp-pair open on $M\otimes_RS_R$, say $\overline{a} \otimes \overline{s}$ is in $\sigma(M\otimes_RS_R)$ but not in $\tau(M\otimes_RS_R)$.

First suppose that $M_R$ is Mittag-Leffler.  Then there is a pp formula $\phi$ such that ${\rm pp}^M(\overline{a})$ is generated by $\phi$ and hence, by \ref{sigmaphi1}(2), such that $_RS_R \models (\sigma: \phi)(\overline{s})$ and also such that $\overline{s} \notin (\tau:\phi)(_RS_R)$.  Then $_RS_R$ satisfies the sentence (in the language of $(R,R)$-bimodules, and where ``$\neg$" means ``not")
$$\exists \overline{y} \,\big( (\sigma:\phi)(\overline{y}) \wedge \neg (\tau: \phi)(\overline{y})\big).$$
Since $R$ is an elementary sub-$(R,R)$-bimodule of $S$, it also satisfies this sentence, so choose $\overline{r}$ from $R$ in $(\sigma:\phi)(_RR_R) \setminus (\tau: \phi)(_RR_R)$.  Then we have $M\simeq M\otimes_RR_R \models \sigma(\overline{a} \otimes \overline{r})$ and also, by choice of $\phi$ and \ref{sigmaphi1}(2), $\overline{a}\otimes \overline{r}$ is not in $\tau(M)$.  That is, $\sigma/\tau$ is open on $M$, as required.

In the second case, $M$ is arbitrary but we have the acc on pp-definable subgroups of $_RR_R$.  Consider the set of pp-definable subgroups of $_RR_R$ of the form $(\tau:\phi)(_RR_R)$ with $\overline{a} \in \phi(M)$.  Since the set of such $\phi$ is closed under finite $\wedge$ and since (from the definitions) we have $(\tau:\phi \wedge \phi') \geq (\tau:\phi) + (\tau: \phi')$, the set of those subgroups $(\tau:\phi)(_RR_R)$ with $\phi \in {\rm pp}^M(\overline{a})$ is upwards-directed so, by assumption, has a maximal element $(\tau:\phi)(_RR_R)$ say.  Since $_RS_R \models \sigma(\overline{a} \otimes \overline{s})$, there is $\phi' \in {\rm pp}^M(\overline{a})$ with $_RS_R \models (\sigma:\phi')(\overline{s})$.  Replacing each of $\phi$ and $\phi'$ by $\phi \wedge \phi'$, we may assume that $\phi=\phi'$.  As before, we therefore obtain $\overline{r} \in (\sigma:\phi)(_RR_R) \setminus (\tau:\phi)(_RR_R)$ and so have $M \models \sigma(\overline{a} \cdot \overline{r})$ (by $\overline{a} \cdot \overline{r}$ we mean $\sum_i \, a_ir_i$).  If we had $M = M\otimes_RR_R \models \tau(\overline{a} \otimes \overline{r})$, then there would be $\psi \in {\rm pp}^M(\overline{a})$ with $\overline{r} \in (\tau:\psi)(_RR_R)$; yet $\overline{r} \notin (\tau:\phi)(_RR_R)$, contradicting maximality of $(\tau:\phi)(_RR_R)$.  Therefore $\overline{a}\otimes \overline{r} = \overline{a}\cdot \overline{r} \notin \tau(M)$ and so $\sigma/\tau$ is open on $M$, as required.
\end{proof}

The second hypothesis is satisfied if $R$ is finitely generated as a module over a noetherian centre (since every pp-definable subgroup is a module over the centre of $R$).  We have the following corollary.

\begin{cor}\label{accdeftens} \marginpar{accdeftens} Suppose that $_RR_R$ has the ascending chain condition on pp-definable subgroups and that $R\prec S$ as $(R,R)$-bimodules.  Let ${\cal D}$ be a definable subcategory of ${\rm Mod}\mbox{-}R$.  Then $({\cal D} \otimes_RS)|_R = {\cal D}$, so ${\cal D}\otimes_RS \subseteq {\cal D}^S$.
\end{cor}

The next corollary moves away from the context of $(R,R)$-bimodules.

\begin{cor}\label{fdalgdeftens} \marginpar{fdalgdeftens} Suppose that $R$ is a finite-dimensional algebra over a field and that $R \prec S$ is an elementary extension {\bf of rings}. Let ${\cal D}$ be a definable subcategory of ${\rm Mod}\mbox{-}R$.  Then $({\cal D} \otimes_RS)|_R = {\cal D}$, so ${\cal D}\otimes_RS \subseteq {\cal D}^S$.
\end{cor}
\begin{proof}  If $S$ is a finite-dimensional algebra, in particular if $S$ is an ultrapower of $R$ then, see \ref{findimup}, this follows from \ref{accdeftens}.  In general $S$ is just an elementary subring of an ultrapower $R^\ast = R^I/{\cal U}$ of $R$.  If, for such a ring $S$, there were $_SS_S$-bimodule pp formulas $\phi_0(\overline{x}), \phi_1(\overline{x}), \dots, \phi_n(\overline{x}), \dots$ such that $\phi_0(_SS_S) < \phi_1(_SS_S) < \dots < \phi_n(_SS_S) < \dots$ were an infinite ascending chain then, for each $n$, we would have that $_SS_S$ satisfies the sentence, from the proof of \ref{elemincl}, in the language of $(S,S)$-bimodules which says that $\phi_n \leq \phi_{n+1}$ but not $\phi_{n+1} \leq \phi_n$.  Since $S \prec R^\ast$ as rings, $R^\ast$ also would satisfy each of these sentences and therefore $_{R^\ast}R^\ast_{R^\ast}$ would have a proper ascending chain of pp-definable subgroups, in contradiction to it (and therefore ${R^\ast}^{\rm op}\otimes R^\ast$) being a finite-dimensional algebra.
\end{proof}

The two different hypotheses of \ref{MotimesSR} allow a common generalisation, as follows.  The proof from \ref{MotimesSR} applies but we present it rather differently so as to give a different perspective on the $R \prec S$ condition.

\begin{prop}\label{attensext} \marginpar{attensext}  Suppose that $R\prec S$ as $(R,R)$-bimodules and that $M_R$ is Mittag-Leffler with respect to $_RS_R$.  Let ${\cal D} = \langle M \rangle$ be the definable subcategory of ${\rm Mod}\mbox{-}R$ generated by $M$.  Then $({\cal D} \otimes_RS)|_R = {\cal D}$, so ${\cal D}\otimes_RS \subseteq {\cal D}^S$.
\end{prop}
\begin{proof} Since $R$ and $S$ have, by \ref{isoup}, isomorphic ultrapowers (as $(R,R)$-bimodules) we have, by \ref{elemextup}, that $S$ is an elementary $(R,R)$-subbimodule of some ultrapower of $_RR_R$, so first we consider the case that $S=R^I/{\cal U}$ for some set $I$ and ultrafilter ${\cal U}$ on $I$.

Suppose that $\sigma/\tau$ is closed on ${\cal D}$; we show that $\sigma /\tau$ is closed on $M\otimes_RS$, allowing us to deduce that $M\otimes_RS \in {\cal D}_S$.

So suppose that $M\otimes_RS \models \sigma(\overline{a}\otimes \overline{s})$ where $\overline{s} = (\overline{r}_i)_i/{\cal U}$ with the $r_i\in R$ and, noting that $\sigma$ is in the language of $R$-modules, we actually have $_RS_R \models (\sigma: \phi)((\overline{r}_i)_i/{\cal U})$, that is $_R(R^I/{\cal U})_R = (_RR_R)^I/{\cal U} \models (\sigma: \phi)((\overline{r}_i)_i/{\cal U})$, for some $\phi \in {\rm pp}^M(\overline{a})$.  So, for some $J\in {\cal U}$, we have $_RR_ R \models (\sigma:\phi)(\overline{r}_j)$ for every $j\in J$.  Therefore $M \simeq M\otimes_RR_R \models \sigma(\overline{a}\otimes \overline{r}_j)$ for every $j\in J$.  Since $\sigma/\tau$ is closed on $M$, we have, for every $j\in J$, $M\otimes_RR_R \models \tau(\overline{a}\otimes \overline{r}_j)$.  Then, for each $j\in J$, $M \models \psi_j(\overline{a})$ and $_RR_R\models (\tau:\psi_j)(\overline{r}_j)$ for some $\psi_j$ such that $\psi_j \in {\rm pp}^M(\overline{a})$.  

By assumption, we may choose $\psi$ such that $(\tau: \psi)(_RR_R)$ is maximal among subgroups of the form $(\tau: \psi)(_RR_R)$ with $\psi \in {\rm pp}^M(\overline{a})$ and hence such that $_RR_R\models (\tau:\psi)(\overline{r}_j)$ for all $j\in J$.  Therefore we have $_R(R^I/{\cal U})_R \models (\tau: \psi)(\overline{s})$, and so $M\otimes_R(R^I/{\cal U})_R \models \tau(\overline{a}\otimes \overline{s})$ and $\sigma/\tau$ is indeed closed on $M\otimes_R (R^I/{\cal U})$.

In the general case, with $S$ an elementary $(R,R)$-subbimodule of $R^I/{\cal U}$, then, when we have $_RS_R\models (\sigma:\phi)(\overline{s})$ and hence $_R(R^I/{\cal U})_R \models (\sigma:\phi)(\overline{s})$, we obtain, by the argument above, that $_R(R^I/{\cal U})_R \models (\tau: \psi)(\overline{s})$ and so, since $S\prec R^I/{\cal U}$, the same is true for $S$, therefore the argument continues as above.

To finish, we recall that if $\sigma/\tau$ is open on $M \in {\cal D}$ then it is open on $M\otimes_RS_R$, by \ref{elem3}.
\end{proof}

\begin{cor} With assumptions as in \ref{attensext}, $M\otimes_RS_R$ generates the same definable subcategory of ${\rm Mod}\mbox{-}R$ as $M_R$.
\end{cor}

\begin{cor}\label{fptensext} \marginpar{fptensext} Suppose that $R \prec S$ as $(R,R)$-bimodules.  Suppose that the definable subcategory ${\cal D}$ of ${\rm Mod}\mbox{-}R$ is generated, as a definable subcategory, by the finitely presented modules in it.  Then $({\cal D} \otimes_RS)|_R = {\cal D}$, so ${\cal D}\otimes_RS \subseteq {\cal D}^S$. 
\end{cor}
\begin{proof} Take $M$ as in \ref{attensext} to be the direct sum of one copy of each module in ${\cal D}\, \cap\, {\rm mod}\mbox{-}R$.
\end{proof}

\begin{example}\label{tensneqext} \marginpar{tensneqext} It is not in general the case that ${\cal D}\otimes_RS = {\cal D}^S$, even if $R \prec S$ as rings.  Indeed consider ${\cal D} = {\rm Mod}\mbox{-}R$ where $R=K\widetilde{A}_1$ is the path algebra of the Kronecker quiver $\widetilde{A}_1$ over an algebraically closed field $K$ and $S=L\widetilde{A}_1$ over some proper algebraically closed extension $L$ of $K$.  (In fact any proper elementary extension of $R$ has this form because the field $K$ is definable as the centre of $R$.)  By \ref{modextn} (or just directly), ${\cal D}^S = {\rm Mod}\mbox{-}S$ which, as a definable subcategory, is generated by the indecomposable finite-dimensional $R$-modules.  The effect of $-\otimes_RS_S$ is just to replace $K$ by $L$ and so the images of these indecomposable modules under $-\otimes_RS_S$ are the preprojectives, the preinjectives and the simple-regular modules with parameter coming from ${\mathbb P}^1(K)$ but, together, these generate a proper definable subcategory of ${\rm Mod}\mbox{-}S$ - proper since it does not contain any simple-regular module with parameter in  $L \setminus K$.
\end{example}

\newpage

\end{document}